\documentclass[a4paper,twoside,english,10pt]{amsart}
\usepackage{lmodern} 
\usepackage[margin=1.3in]{geometry} 
\usepackage{graphicx}             
\usepackage{amsmath}              
\usepackage{amsfonts}
\usepackage{amsthm}               
\usepackage{verbatim}             
\usepackage{indentfirst}           
\usepackage{underscore}
\usepackage{enumitem}
\usepackage{mathtools} 
\usepackage{hyperref}
\usepackage{amssymb}
\usepackage{xcolor}
\usepackage[all]{xy}
\usepackage[utf8]{inputenc}
\makeatletter
\def\l@section{\@tocline{1}{4pt}{1pc}{}{}}
\def\l@subsection{\@tocline{2}{4pt}{2.5pc}{5pc}{}}
\def\l@subsubsection{\@tocline{3}{4pt}{4pc}{5pc}{}} 
\makeatother

\makeatletter
\let\orig@contentsline\contentsline
\renewcommand{\contentsline}[4]{%
  \begingroup
    \def\@tempa{#1}\def\@tempb{section}%
    \ifx\@tempa\@tempb \let\Hy@linktoc\Hy@linktoc@section
    \else \let\Hy@linktoc\Hy@linktoc@none \fi
    \orig@contentsline{#1}{#2}{#3}{#4}%
  \endgroup}
\newcommand{\toc@numlink}[1]{%
  \ifx\indentlabel\@tochangmeasure #1%
  \else\Hy@toclinkstart#1\Hy@toclinkend\fi}
\renewcommand{\tocsubsection}[3]{%
  \indentlabel{\@ifnotempty{#2}{\ignorespaces#1 \toc@numlink{#2}.\quad}}#3}

\makeatother

\newtheorem{thm}{Theorem}
\newtheorem{lem}[thm]{Lemma}
\newtheorem{prop}[thm]{Proposition}
\newtheorem{cor}[thm]{Corollary}

\newtheorem{question}[thm]{Question}
\theoremstyle{remark}
       
\theoremstyle{remark}

       \theoremstyle{definition} 
\newtheorem{defin}[thm]{Definition}

\theoremstyle{remark}
\newtheorem{remark}[thm]{Remark}

\numberwithin{thm}{section}
\numberwithin{equation}{thm}

\newcommand{\Mod}[1]{\ (\textup{mod}\ #1)}

  \newcommand{\SL}{\operatorname{SL}}
    \newcommand{\GL}{\operatorname{GL}}

  \newcommand{\Hom}{\operatorname{Hom}}

\newcommand{\Pic}{\operatorname{Pic}} 
\newcommand\scalemath[2]{\scalebox{#1}{\mbox{\ensuremath{\displaystyle #2}}}}  
\usepackage{amssymb,fge}

\title[growth in partially commutative semigroups of maps]{arithmetic growth in partially commutative semigroups of endomorphisms}
\author{Wade Hindes}
\begin{document}
\begin{abstract} We establish asymptotics for the degree and height growth rate functions of partially commutative semigroups of endomorphisms of varieties, i.e., semigroups where the only relations between generators, if any, are commuting ones. To do this, we prove a general result connecting the growth rates of two abstract functions, one multiplicative and one simply nonnegative, subject to a Tate telescoping hypothesis and a sufficiently nice normal form on the underlying semigroup. Finally, we construct several new families of endomorphisms on which our asymptotics apply. To do this, we generalize the standard ping-pong lemma for elements to a ping-pong lemma for subsemigroups.           
\end{abstract}
\maketitle
\tableofcontents
\section{Introduction}\label{sec:intro}
Given a variety $V$ defined over a number field $K$ and a suitable logarithmic height function $h: V(\overline{K})\rightarrow \mathbb{R}$, the growth rate of the counting function, 
\begin{equation}\label{eq:h+count+variety}
\#\{P\in V(K)\,: h(P)\leq x\},
\end{equation} 
can be a useful tool for detecting important invariants of $V$ and $K$. For example, when $V$ is projective space \cite{schanuel1979heights}, then \eqref{eq:h+count+variety} detects the dimension of $V$ and the class number, the regulator, and the number of roots of unity of $K$. When $V$ is an abelian variety \cite{neron1965quasi}, then  \eqref{eq:h+count+variety} detects the rank of the Mordell-Weil group, the size of the torsion subgroup, and the regulator of $V(K)$. When $V$ is a curve, then \eqref{eq:h+count+variety} detects the genus of $V$. For additional examples, including the Batyrev-Manin conjecture, see  \cite{baragar1996rational,batyrev1990nombre,franke11996rational,peyre1995hauteurs,silverman1991rational}. On the other hand, a general heuristic in arithmetic dynamics is that many of the tools, theorems, and conjectures in arithmetic geometry have dynamical analogs, once we replace the arithmetic object (such as the set $V(K)$ of rational points) with a suitable dynamical object (such as a dynamical orbit); see \cite{benedetto2019current} for a discussion of several examples of this heuristic, including problems related to uniform boundedness, Galois representations, moduli space, and unlikely intersections.

To illustrate this analogy, it was shown in \cite{MR4988033} that for \emph{free} polarized semigroups of endomorphisms, dynamical orbits exhibit similar height-growth to rational points on abelian varieties: namely \eqref{eq:h+count+variety} and its dynamical counterpart grow like $cx^{\rho}$ for some $c,\rho>0$. Moreover, on the projective line at least, one should expect most semigroups to be free; see, for instance, \cite{beaumont2025uniformtitsalternativeendomorphisms,MR4780496} and \cite[Proposition 19]{hindes2023size}. On the other hand, there are many naturally occurring dynamical semigroups that are not free --- for example, both projective space and abelian varieties exhibit many commuting maps. Moreover, the study of commuting maps has a long history in dynamics, going back to Julia \cite{julia1922memoire}, Ritt \cite{ritt1923permutable}, and Fatou \cite{fatou1923substitutions}. In particular, in this paper we study a dynamical analog of \eqref{eq:h+count+variety} for partially commutative semigroups, i.e., finitely generated semigroups where the only relations between generators, if any, are commuting ones. 

To state our main result, we fix some notation. In what follows, $\mathcal{S}=\langle f_1,\dots, f_r\rangle$ is a finitely generated semigroup of finite endomorphisms of a variety $V$ under composition. Moreover, we assume that $\mathcal{S}$ is \textbf{\emph{partially commutative}}, that is, $\mathcal{S}$ is isomorphic to the free semigroup generated by the alphabet of $f$'s modulo a finite collection of relations of the form $f_i\circ f_j=f_j\circ f_i$ for some distinct $i$ and $j$. In particular, we can associate to $\mathcal{S}$ a finite simple graph $\mathcal{G}=\mathcal{G}(\mathcal{S})$ with vertices $f_1,\dots, f_r$ and edges $\{f_i,f_j\}$ with $i\neq j$ whenever $f_i$ and $f_j$ commute. We call $\mathcal{G}$ the \textbf{\emph{commutation graph}} of $\mathcal{S}$. Likewise, we let $\overline{\mathcal{G}}$ denote the complement graph of $\mathcal{G}$ so that $\overline{\mathcal{G}}$ has the same vertex set as $\mathcal{G}$, has a loop at every vertex, and has an edge between distinct $f_i$ and $f_j$ whenever $f_i$ and $f_j$ do not commute. 

Furthermore, we let $\mathcal{L}$ denote an ample invertible sheaf on $V$ and let $h_{\mathcal{L}}$ denote an associated Weil height function. We say that $\mathcal{S}$ is \textbf{\emph{polarized}} by $\mathcal{L}$ if for all $F\in\mathcal{S}$ there exists an integer $d(F)\geq2$ such that $F^*(\mathcal{L})\cong L^{\otimes d(F)}$. In particular, we note that when $V=\mathbb{P}^N$ is projective space, every semigroup of non-linear endomorphisms is polarized. 

Finally, we say that a point $P\in V$ is \textbf{\emph{preperiodic}} for $\mathcal{S}$ if there exist $g\in\mathcal{S}$ and $f\in \mathcal{S}\cup\{\textup{Id}_V\}$ such that $g\circ f(P)=f(P)$. Then we have the following theorem:            
\begin{thm}\label{thm:main} Let $(\mathcal{S},\mathcal{G})$ be a finitely generated partially commutative semigroup and let $d: \mathcal{S}\rightarrow \mathbb{Z}_{\geq2}$ be a multiplicaitve function. 
Moreover, assume that the complement graph $\overline{\mathcal{G}}$ is connected and that the image $d(\mathcal{S})$ is not contained in a cyclic semigroup. Then there are positive constants $c,\rho>0$, depending on $\mathcal{S}$ an $d$, such that 
\begin{equation}\label{d-count}
\#\{F\in \mathcal{S}\,: d(F)\leq x\}=cx^{\rho}+o(x^\rho).
\end{equation}
In particular, if $\mathcal{S}$ is a semigroup of endomorphisms on a variety $V$, then \eqref{d-count} holds for all of the following examples of $d$:
 \vspace{.15cm}
 \begin{enumerate}
 \item[\textup{(1)}] $\mathcal{S}$ is generated by finite maps of degree at least $2$ and $d(F)=\deg(F)$. 
 \vspace{.15cm}
 \item[\textup{(2)}] $\mathcal{S}$ is polarized by $\mathcal{L}$ and $d(F)\geq2$ is given by  $F^*(\mathcal{L})\cong L^{\otimes d(F)}$.
 \vspace{.15cm}
 \item[\textup{(3)}] $\mathcal{S}$ is generated by finite maps with a common super-attracting fixed point $P\in V$ and $d(F)=\deg_P(F)$ is the local degree at $P$.
 \vspace{.15cm}
 \end{enumerate}
 Furthermore, when $\mathcal{S}$ is polarized and $\mathcal{S}$ and $\mathcal{L}$ are defined over a number field $K$, then for all non-preperioidc points $P\in V(\overline{K})$ there is a constant $c(P,\mathcal{L})>0$ such that
 \vspace{.15cm}
 \begin{equation}\label{h-count}
\#\{F\in \mathcal{S}\,: h_{\mathcal{L}}(F(P))\leq x\}=c(P,\mathcal{L})x^{\rho}+o(x^{\rho}).
 \vspace{.15cm}
 \end{equation} 
 Here $h_{\mathcal{L}}: V(\overline{K})\rightarrow\mathbb{R}_{\geq0}$ is a Weil height function associated to $\mathcal{L}$ and $\rho$ is associated to the multiplicitive function $d$ given in \textup{(2)}.
 \end{thm}
\begin{remark} When $\mathcal{S}$ as in Theorem \ref{thm:main} is a free semigroup, the commutation graph $\mathcal{G}=E_r$ is the empty graph with $r\geq2$ vertices and no edges. Hence, the complement graph $\overline{\mathcal{G}}=K_r$ is the complete graph with $r$ vertices. In particular, $\overline{\mathcal{G}}$ is connected in this case, and so we recover the main statement of \cite[Theorem 1.1]{MR4988033}.         
\end{remark}
\begin{remark} We note that if $\mathcal{S}$ is any semigroup of finite maps of degree at least $2$ on a topological space $V$, then the degree count \eqref{d-count} gives an \emph{intrinsic} measure of the growth rate of the semigroup $\mathcal{S}$, i.e., a growth rate that does not depend on the generating set. In particular, \eqref{d-count} may be of independent interest, not just to those experts studying heights.     
\end{remark}

\begin{remark}
Partially commutative semigroups (and monoids) appear in many contexts, and we mention a few, with no attempt at completeness. They were introduced by Cartier and Foata~\cite{MR239978} in their study of rearrangements of words. Later, Viennot's theory of heaps of pieces~\cite{Viennot1986} gave them a geometric model that has become a useful tool in enumerative combinatorics. In computer science, they are known as \emph{trace monoids}. Mazurkiewicz~\cite{Mazurkiewicz1977} proposed them as a model of concurrent processes, in which commuting letters represent actions that can occur independently, and they have since become a standard tool in the theory of concurrency (see, for instance, Diekert's monograph~\cite{Diekert1990}). In semigroup theory, they are called \emph{free partially commutative monoids}
or \emph{graph monoids}, and they are the basic examples of graph products of monoids; their algebraic properties, such as conjugacy, Green's relations, and the word problem, have been studied in~\cite{Duboc1986, DuchampKrob1993, FountainKambites2009,
daCosta2001} and in many other places. They also appear in topology and dynamics: for example, McMullen~\cite{McMullen2015} used the weighted growth rates of these objects, which he calls right-angled Artin semigroups, to bound spectral radii of Perron--Frobenius matrices. Further references are given below where they are used.
\end{remark}
Next we classify the possible partially commutative semigroups that can be obtained by composing non-linear rational maps of the projective line.   
\begin{thm}\label{thm:classification+intro} Let $K$ be a field of characteristic zero and let $\mathcal{S}$ be a finitely generated partially commutative semigroup. Then $\mathcal{S}$ can be embedded as a semigroup of non-linear rational functions in K(x) under composition if and only if $\mathcal{S}$ is a free product of free commutative semigroups (i.e., $\mathcal{G}(\mathcal{S})$ is a disjoint union of complete subgraphs).        
\end{thm}
\begin{remark} The main difficulty in proving Theorem \ref{thm:classification+intro} is the construction of arbitrary  free products of free commutative semigroups; see Proposition \ref{prop:Free+products+free+abelian} below. In particular, to do this we combine a ping-pong lemma for subsemigroups (see Lemma \ref{lem:ping+pong}) with Mason's abc-theorem over polynomial rings.     
\end{remark}
Combining the two theorems above, we have the following consequence, which implies that the height growth rates in Theorem \ref{thm:main} apply to nearly all partially commutative semigroups acting on the projective line.   
\begin{cor}\label{thm:ht+bounds+non-commutative} Let $K$ be a field of characteristic zero and let $\mathcal{S}\subset K(x)$ be a finitely generated partially commutative semigroup of non-linear rational functions under composition. If $\mathcal{S}$ is not commutative and $\deg(\mathcal{S})$ is not contained in a cyclic semigroup, then the degree and height asymptotics in Theorem \ref{thm:main} hold for $\mathcal{S}$.  
\end{cor}
\begin{proof} Let $\mathcal{S}$ be as in the statement. Then $\mathcal{G}(\mathcal{S})=K_{m_1}\sqcup\dots\sqcup K_{m_r}$ for some $r,m_1\dots,m_r\geq 1$ by Theorem \ref{thm:classification+intro}. In particular, since taking complements of graphs exchanges disjoint unions with joins, we see that
\[\overline{\mathcal{G}(\mathcal{S})}=\overline{K_{m_1}\sqcup\dots\sqcup K_{m_r}}=\overline{K_{m_1}}\vee\dots\vee \overline{K_{m_r}}=E_{m_1}\vee\dots\vee E_{m_r}.\]
Thus, $\overline{\mathcal{G}(\mathcal{S})}$ is connected unless $r=1$. But this means $\mathcal{G}(\mathcal{S})=K_{m_1}$ and  $\mathcal{S}$ is commutative. Therefore, we deduce that if $\mathcal{S}\subset K(x)$ is partially commutative but not commutative, $\mathcal{S}$ consists of non-linear maps, and $\deg(\mathcal{S})$ is not contained in a cyclic semigroup, then Theorem \ref{thm:main} holds for $\mathcal{S}$.      
\end{proof}
Finally, we illustrate that in higher dimensions, there are many more examples of polarized partially commutative semigroups of endomorphisms than just free products of free commutative semigroups. In particular, we highlight the following example using products of elliptic curves.  
\begin{prop}\label{prop:En+vee+Km} Let $K$ be a number field, let $E$ be an elliptic curve defined over $K$, and let $V=E\times E\times E$. Fix any integers $n,m\geq1$. Then there exists a polarizable semigroup $\mathcal{S}$ of endomorphisms of $V$ such that $\mathcal{S}$ is partially commutative and $\mathcal{G}(S)=E_n\vee K_m$. Hence, $S$ is a direct product of a free semigroup of rank $n$ with a free commutative semigroup of rank $m$.  
\end{prop}
\begin{remark} We note that if $\mathcal{S}$ is as in Proposition \ref{prop:En+vee+Km}, then the complement graph $\overline{\mathcal{G}(\mathcal{S})}$ is not connected; hence, the conclusions of Theorem \ref{thm:main} do not follow directly. Nevertheless, numerical evidence suggests that the degree and height asymptotics in Theorem \ref{thm:main} still hold for these semigroups. In particular, it is possible that the condition that $\overline{\mathcal{G}(\mathcal{S})}$ be connected can be weakened. On the other hand, there is certainly some condition on the commutation graph $\mathcal{G}(\mathcal{S})$ or its complement that is necessary. For example, if $\mathcal{G}(\mathcal{S})=K_m$ is a complete graph (so that $\overline{\mathcal{G}(\mathcal{S})}=E_m$), then the degree and height growth rates in $\mathcal{S}$ are $o(x^\rho)$ for all $\rho>0$. In particular, the problem of finding necessary and sufficient criteria for obtaining the asymptotics in Theorem \ref{thm:main} is an interesting one, and we plan to return to this in the future.             
\end{remark}
\begin{remark} On the other hand, by combining the constructions in Proposition \ref{prop:Free+products+free+abelian} and Propositions \ref{prop:abelian+examples} below, one may obtain semigroups of split maps of $V\times\mathbb{P}^1$ that satisfy all of the hypotheses of Theorem \ref{thm:main}. In particular, the height asymptotics hold for these examples. Moreover, these examples are not free products of free commutative semigroups, and so in higher dimensions, Theorem \ref{thm:main} applies to more semigroups than in dimension one.   
\end{remark}
An outline of our paper is as follows. In Section \ref{sec:semigroups}, we set some basic notation and sketch the usefulness of (suffix-closed) normal forms when comparing degrees to heights. In Section \ref{sec:d+to+h}, we prove a general result connecting the growth rates of two abstract functions, one multiplicative and one simply nonnegative, subject to a Tate telescoping hypothesis and a suffix-closed normal form. In Section \ref{sec:counting+walks}, we count walks on a directed graph of bounded multiplicative weight and given terminal segment, which will be used to count functions of bounded degree in a given suffix class; see hypothesis (3) of Theorem \ref{thm:d+to+h}. Next, we outline the basic theory of partially commutative semigroups and prove our main result, Theorem \ref{thm:main}, in Section \ref{sec:trace}. Here, the Cartier-Foata normal form \cite{MR239978} provides the bridge to the results in both Section \ref{sec:d+to+h} and Section \ref{sec:counting+walks}; see Proposition \ref{prop:CF+implies+suffix+closed+directed+graph+walks}. Finally, in Section \ref{sec:representations}, we prove the classification result for partially commutative semigroups in dimension one, as well as construct more examples in higher dimensions. \\[5pt] 
\indent \textbf{Acknowledgments:} We thank the American Institute of Mathematics for hosting the workshop, the Dynamics of Multiple Maps, in November 2025 which inspired some of the ideas in this paper. Moreover, we used AI to search and synthesize the literature on partially commutative semigroups and on counting weighted walks on a directed graph, and this information was incorporated in Sections \ref{sec:counting+walks} and \ref{sec:trace}.         

\section{Semigroups with Normal Forms}\label{sec:semigroups} 
\subsection{Preliminaries}
We begin with a discussion of some basic ideas and notation in the theory of semigroups; see \cite{HowieIntro, howie_fundamentals} for more complete introductions. For elements $f_1,\dots,f_n$ of a semigroup $\mathcal{S}$, we write $\langle f_1,\dots,f_n\rangle$ for the subsemigroup they generate, i.e., the set of all products $f_{i_1}\cdots f_{i_k}$ with $k\geq 1$. In the case when $\mathcal{S}$ is a set of self-maps of a variety $V$, the operation is composition. Also, in accordance with the literature, we write $\mathcal{S}^1=\mathcal{S}\sqcup\{1\}$ for the monoid
obtained by adjoining a new element $1$ to $\mathcal{S}$, with the multiplication of $\mathcal{S}$ extended by $1s=s1=s$ for all $s\in \mathcal{S}^1$. 
\begin{remark}
We adjoin a new identity even if $\mathcal{S}$ already has one; a common alternative convention sets $\mathcal{S}^1=\mathcal{S}$ in that case, but the descriptions of free and direct products below require the new identity. This makes no difference for polarized semigroups, which contain no idempotents: if $e^2=e$, then $d(e)^2=d(e)$, contradicting $d(e)\geq 2$. In particular, if $\mathcal{S}$ is a semigroup of self-maps of $V$ not containing the identity map $\mathrm{id}_V$ (for instance, a polarized
one), then $\mathcal{S}\cup\{\mathrm{id}_V\}$ is a copy of $\mathcal{S}^1$, and we take
$1=\mathrm{id}_V$.
\end{remark}

Let $\mathcal{S}_1,\dots,\mathcal{S}_n$ be semigroups (replaced by disjoint copies if necessary). Then the \emph{free product} $\mathcal{S}_1*\cdots*\mathcal{S}_n$ is the coproduct in the category of semigroups (see, e.g., \cite[Ch.~VII]{HowieIntro}); equivalently, it is the monoid free product of $\mathcal{S}_1^1,\dots,\mathcal{S}_n^1$ with its identity removed. Concretely, its elements are the nonempty finite words $s_1\cdots s_k$ with letters in $\mathcal{S}_1\cup\cdots\cup \mathcal{S}_n$ and consecutive letters in different factors. Words are multiplied by concatenation, except that when the last letter (from left to right) of the first word and the first letter of the second word lie in the same $\mathcal{S}_i$, these two letters are replaced by their product in~$\mathcal{S}_i$. Similarly, we define the \emph{direct product}, 
\begin{equation}\label{eq:dir+product}
  \mathcal{S}_1\times\cdots\times \mathcal{S}_n
  =\bigl(\mathcal{S}_1^1\times\cdots\times \mathcal{S}_n^1\bigr)\setminus\{(1,\dots,1)\}
\end{equation}
with componentwise multiplication, i.e., the monoid direct product of $\mathcal{S}_1^1,\dots,\mathcal{S}_n^1$ with its identity removed; note that this is not the Cartesian product of the $\mathcal{S}_i$ --- with the definition of the direct product above, each $\mathcal{S}_i$ embeds in $\mathcal{S}_1\times\cdots\times \mathcal{S}_n$ as the tuples whose other coordinates are all equal to $1$. In particular, the embedded copies of the $\mathcal{S}_i$ commute with each other and generate the direct product. Equivalently, $\mathcal{S}_1\times\cdots\times \mathcal{S}_n$ is the quotient of the free product $\mathcal{S}_1*\cdots*\mathcal{S}_n$ by the relations $st=ts$ for $s\in \mathcal{S}_i$, $t\in \mathcal{S}_j$, $i\neq j$. 

Now let $\mathcal{S}_1,\dots,\mathcal{S}_n$ be subsemigroups of a semigroup $\mathcal{S}$. Then the inclusions $\mathcal{S}_i\subseteq\mathcal{S}$ induce a surjective homomorphism $\mathcal{S}_1*\cdots*\mathcal{S}_n\to\langle \mathcal{S}_1\cup\cdots\cup \mathcal{S}_n\rangle$.
If moreover $st=ts$ whenever $s\in \mathcal{S}_i$, $t\in \mathcal{S}_j$ and $i\neq j$, then $(s_1,\dots,s_n)\mapsto s_1\cdots s_n$, where coordinates equal to $1$ are omitted, defines a surjective homomorphism
$\mathcal{S}_1\times\cdots\times \mathcal{S}_n\to\langle \mathcal{S}_1\cup\cdots\cup \mathcal{S}_n\rangle$. It is an isomorphism exactly when each element of $\langle \mathcal{S}_1\cup\cdots\cup \mathcal{S}_n\rangle$
can be written in only one way as such a product. Next, for convenience, we say that $\mathcal{S}=\langle f_1,\dots,f_n\rangle$ is \emph{free} if the homomorphism $\langle f_1\rangle*\cdots*\langle f_n\rangle\to \mathcal{S}$ is an isomorphism. Thus, freeness is a property of the given generators $f_1,\dots,f_n$, not merely of the
abstract semigroup $\mathcal{S}$. When $\mathcal{S}$ is polarized, each $\langle f_i\rangle$ is isomorphic to $\mathbb{Z}_{\geq1}$ with addition. Hence $\mathcal{S}$ is free exactly when distinct words in $f_1,\dots,f_n$ define distinct maps,
that is, when $f_1,\dots,f_n$ is a free basis of $\mathcal{S}$.

Next, we record notation for two families of semigroups that will serve as building blocks for examples to come. Fix $n\geq1$ and distinct symbols $x_1,\dots,x_n$. Then the \emph{free semigroup of rank $n$} is
\[
  \mathcal{F}(n)=\langle x_1\rangle*\cdots*\langle x_n\rangle,
\]
the free product of $n$ copies of $\mathbb{Z}_{\geq1}$ with addition; concretely, its elements are the nonempty finite words in $x_1,\dots,x_n$ under concatenation. The \emph{free
commutative semigroup of rank $n$} is
\[
  \mathcal{A}(n)=\langle x_1\rangle\times\cdots\times\langle x_n\rangle,
\]
the direct product \eqref{eq:dir+product} of $n$ copies of $\mathbb{Z}_{\geq1}$; concretely, recording an element $x_1^{a_1}\cdots x_n^{a_n}$ by its exponent vector identifies
\[
  \mathcal{A}(n)=\mathbb{Z}_{\geq0}^{\,n}\setminus\{(0,\dots,0)\}
\]
with addition. Note that $\mathcal{F}(1)=\mathcal{A}(1)=\mathbb{Z}_{\geq1}$, and that in both cases the indecomposable elements are exactly $x_1,\dots,x_n$, so the rank $n$ is determined by the isomorphism class.

Finally, we note that these semigroups are characterized by the expected universal properties: given any semigroup $\mathcal{S}$ and any $f_1,\dots,f_n\in\mathcal{S}$, there is a unique homomorphism $\mathcal{F}(n)\to\mathcal{S}$ with $x_i\mapsto f_i$, and if the $f_i$ commute pairwise there is a unique homomorphism $\mathcal{A}(n)\to\mathcal{S}$ with $x_i\mapsto f_i$. In particular, if $\mathcal{S}=\langle f_1,\dots,f_n\rangle$ is polarized, then $\mathcal{S}$ is free in the sense above exactly when $x_i\mapsto f_i$ defines an isomorphism 
$\mathcal{F}(n)\xrightarrow{\ \sim\ }\mathcal{S}$. We say $\mathcal{S}$ is \emph{free commutative} on $f_1,\dots,f_n$ when the $f_i$ commute pairwise and $x_i\mapsto f_i$ defines an isomorphism $\mathcal{A}(n)\xrightarrow{\ \sim\ }\mathcal{S}$.
\subsection{Strategy Sketch}\label{subsec:sketch} 
In general, the strategy in this paper is to use the degree growth rate function on a semigroup $\mathcal{S}$ of finite self maps, 
\begin{equation}\label{eq:degree+growth}
\#\{F\in \mathcal{S}\, :\,\deg(F)\leq x\},
\end{equation}
to control the height growth rate function on an orbit,
\begin{equation}\label{eq:height+growth}
\#\{F\in \mathcal{S}\,:\, h(F(P))\leq x\}.
\end{equation}
This approach is in keeping with a general heuristic in arithmetic dynamics and made precise in the case of iterating a single function by the Kawaguchi-Silverman conjecture \cite{KS}; see also \cite{KSConj1,KSConj2,KSConj3,KSConj4,KSConj5} for established cases of this conjecture. Furthermore, this strategy is reasonable for two practical reasons: first, heights tend to grow by a factor of the degree,
\begin{equation}\label{ht+degree+functoriality}
h(F(P))=\deg(F)h(P)+O(1),
\end{equation} 
where $O(1)$ depends on $F$ but not on $P$, and second, the degree map is multiplicative, so that \eqref{eq:degree+growth} is likely easier to understand than \eqref{eq:height+growth}. To make this idea a bit more precise, let $\psi(F):=h(F(P))/\deg(F)$ denote the ratio of the height $h$ evaluated at $F(P)$ and $\deg(F)$, a truncated canonical height function; see \cite[\S3.4]{SilvDyn}. Then, similar to \cite[Lemma 3.2]{MR4780496}, it follows from Tate's telescoping argument and \eqref{ht+degree+functoriality}, that 
\[|\psi(F_1)-\psi(F_2)|\leq \frac{C(\mathcal{S})}{\deg(G)}\]
whenever $F_1,F_2\in \mathcal{S}_G:=\{F'\circ G: F'\in\mathcal{S}^1\}$ are in the same right coset for some $G\in \mathcal{S}$; here, $C(\mathcal{S})$ is a constant depending on $\mathcal{S}$. In particular, by assuming that $\deg(G)$ is sufficiently large, we deduce that the ratio function \textbf{\emph{$\psi$ is approximately constant on cosets}} and write $\psi(F)\approx\psi(G)$ for all $F\in\mathcal{S}_G$ to convey this is in our sketch. Hence, if we can find a suitable finite transversal $\mathcal{T}\subset\mathcal{S}$ such that $\mathcal{S}=\bigsqcup_{G\in\mathcal{T}}\mathcal{S}_G$ is a disjoint union of cosets (or nearly so), then 
\begin{equation}\label{eq:sketch}
\begin{split} 
\#\{F\in \mathcal{S}\,:\, h(F(P))\leq x\}&=\sum_{G\in\mathcal{T}}\#\{F\in\mathcal{S}_G\,:\, h(F(P))\leq x\}\\[5pt] 
&\approx \sum_{G\in\mathcal{T}}\#\Big\{F\in\mathcal{S}_G\,:\, \deg(F)\leq \frac{x}{\psi(G)}\Big\}.
\end{split} 
\end{equation}
Therefore, if we can find an asymptotic (with continuous main term) for the growth rate of the multiplicative function $d(F)=\deg(F)$ on all cosets $\mathcal{S}_G\subset \mathcal{S}$, not just on the entire semigroup $\mathcal{S}$, then we can obtain an asymptotic for the growth rate of $h(F(P))$ on all of $\mathcal{S}$ as desired. 

In fact, for free semigroups this sketch can be made precise using ideas in \cite[Lemma 2.12]{MR4988033}: in this case, right cosets defined by elements $G\in\mathcal{S}$ of sufficiently large length gives the desired decomposition 
$\mathcal{S}=\bigsqcup_{G\in\mathcal{T}}\mathcal{S}_G$. 
However, for non-free semigroups such a coset decomposition is not guaranteed. Nevertheless, it is possible to adapt this argument for semigroups equipped with a sufficiently nice normal form (i.e., unique way to represent elements of $\mathcal{S}$ by words), where suffixes (e.g., $G$ in factorization $F=F'\circ G\in\mathcal{S}_G$) are well defined.         
\subsection{Normal Forms} 
With the strategy above in mind, we fix some notation. Let $\mathcal{S}$ be an abstract semigroup generated by a finite set $\Gamma$, let $\Gamma^*$ denote the free monoid (including the empty word $\epsilon$) generated by $\Gamma$, and let $\Gamma^+=\Gamma\setminus\epsilon$. Given a word $w\in \Gamma^*$, let $|w|$ denote the length. Let $\mathcal{S}^1=\mathcal{S}\cup\{1\}$ where $1$ is an identity element and let $\pi:\Gamma^*\rightarrow \mathcal{S}^1$ be the canonical surjective monoid homomorphism with $\pi(\epsilon)=1$. Moreover, we write $\overline{w}=\pi(w)$ for $w\in\Gamma^*$ to ease notation. 

We now discuss suffixes. Given $v,w\in\Gamma^*$, we say that $v$ is a \emph{suffix} of $w$, denoted by $v\leq w$,  if $w=uv$ for some $u\in \Gamma^*$. Moreover, we say that $v$ is a proper suffix if $u\neq \epsilon$. In particular, it is clear that $\leq$ is a partial order on $\Gamma^*$, and $\leq$ is both transitive and reflexive. Similarly, we define $\textup{Suff}(w)=\{v\in\Gamma^*: v\leq w\}$ to be the set of suffixes of $w$, and note that $\textup{Suff}(w)$ is a finite set with $|w|+1$ elements. With these basic notions in place, we define a new ingredient in our study of degrees and heights, suffix-closed normal forms.
\begin{defin}\label{def:cross-section} A subset $\mathcal{N}\subseteq\Gamma^*$ is called a set of \emph{normal forms} (or transversal, cross-section) of $\mathcal{S}$ if the restriction map $\pi|_{\mathcal{N}}:\mathcal{N}\rightarrow \mathcal{S}^1$ is a bijection. Moreover, we write $\textup{nf}: \mathcal{S}^1\rightarrow\mathcal{N}$ for the inverse map of $\pi$ when restricted to $\mathcal{N}$.   
\end{defin}
\begin{defin}\label{def:suffix-closed} A subset $\mathcal{N}\subseteq\Gamma^*$ is called \emph{suffix-closed} if $w\in\mathcal N$ and $v\leq w$, then $v\in\mathcal{N}$. In particular, if $\mathcal{N}$ is a set of normal forms that is also suffix-closed, then we call $\mathcal{N}$ a \emph{suffix-closed normal form} on $\mathcal{S}$.      
\end{defin}
\begin{remark}\label{rmk:dependence+generating+set} We note that a suffix-closed normal form $\mathcal{N}$ depends on the generating set $\Gamma$ of $\mathcal{S}$. However, we do not include this dependence in our notation.     
\end{remark}
\begin{remark} The Cartier-Foata normal form in a partially commutative semigroup gives a suffix-closed normal form, when we view the set of non-empty cliques as a generating set $\Gamma$ (instead of the original generating letters); see Proposition \ref{prop:CF+implies+suffix+closed+directed+graph+walks} below.    
\end{remark}
However, in order to use $\mathcal{N}$ to count degrees and heights, we need to stipulate further properties on $\mathcal{N}$. In particular, with the proofs in \cite[\S2.2]{MR4988033} in mind, we use $\mathcal{N}$ to define suffix elements in $\mathcal{S}$ (not just in $\Gamma^*$) and suffix deletion in $\mathcal{S}$. We make this process precise with the following result. 
\begin{lem}\label{lem:suffix+deletion}
Let $\mathcal{N}$ be a suffix-closed normal form on $\mathcal{S}$ and let $F\in \mathcal{S}^1$ with $|\textup{nf}(F)|=n$. For $0\leq i\leq n$, let $s_i(F)$ be the suffix of $\textup{nf}(F)$ of length $i$. Then: 
\vspace{.1cm} 
\begin{enumerate}
    \item[\textup{(1)}] We have that $s_i(F)\in\mathcal{N}$ for every $1\leq i\leq n$. Moreover, the element $F^{(i)}:=\overline{s_i(F)}\in \mathcal{S}^1$ satisfies $\textup{nf}(F^{(i)})=s_i(F)$.
    \vspace{.15cm} 
    \item[\textup{(2)}] The map $i\rightarrow F^{(i)}$ is injective with $F^{(0)}=1$ and $F^{(n)}=F$. 
    \vspace{.15cm} 
    \item[\textup{(3)}] For all $i\leq j$, there is a $U\in S^1$ with $F^{(j)}=UF^{(i)}$. Specifically, if we write $s_j(F)=us_i(F)$ for some unique $u\in \Gamma^*$, then $U=\overline{u}$. \vspace{.15cm}    
    \item[\textup{(4)}] We have that $\textup{Suff}(\textup{nf}(F^{(j)}))\subseteq \textup{Suff}(\textup{nf}(F))$ and $F^{(i)}=(F^{(j)})^{(i)}$ for all $i\leq j$.   
\end{enumerate}
\end{lem}
\begin{proof} Statement (1) follows immediately from suffix-closure and Definition \ref{def:cross-section}: $s_i(F)$ is a normal word, and so it is the normal form of the element it represents. For (2), we note that for $0\leq i\neq j\leq n$, we have that $|s_i(F)|=i\neq j=|s_j(F)|$. Hence, $s_i(F)$ and $s_j(F)$ are distinct words in $\mathcal{N}$. But, since $\pi$ is injective on $\mathcal{N}$ (since it is a set of normal forms), we have that $\overline{s_i(F)}\neq \overline{s_j(F)}$. Statement (3) follows from the fact that $\pi$ is a morphism. Finally, statement (4) is the transitivity of $\leq$ together with statement (1).      
\end{proof}
In light Lemma \ref{lem:suffix+deletion}, we can define the deletion of elements of $\mathcal{S}$ --- we call the elements $F^{(i)}$ for $0\leq i\leq |\textup{nf}(F)|$ the \emph{suffixes} of $F$ and $F^{-}:=F^{|\textup{nf}(F)|-1}$, the \emph{first letter deletion} of $F$. We now bring the multiplicative function $d$ into the picture and collect a few facts about $d$ under deletion.
\begin{lem}\label{lem:d+deletion} With the same notation as above, let $d:\mathcal{S}\rightarrow\mathbb{Z}_{\geq2}$ be a multiplicative function and extend $d$ to $\mathcal{S}^1$ by setting $d(1)=1$. Let $F\in \mathcal{S}^1$, let $n:=|\textup{nf}(F)|$, and let $0\leq i\leq j\leq n$. Then $d(F^{(i)})$ divides $d(F^{(j)})$ and     \[1=d(F^{(0)})<d(F^{(1)})<\cdots<d(F^{(n)})=d(F),\qquad d(F^{(i)})\ge 2^{i}.\] 
\end{lem}
\begin{proof} This follows easily from the fact that $d$ is multiplicative and part (3) of Lemma \ref{lem:suffix+deletion}.     
\end{proof}
Next, we define the crossing suffix at a particular level of $d$. Namely, for $L>1$, let $\mathcal{S}_{\geq L}$ be the set of $F\in \mathcal{S}$ with $d(F)\geq L$. Then for $F\in \mathcal{S}_{\geq L}$, we define  
\begin{equation*}\label{eq:sigma}
\sigma_L(F):=F^{(i_L(F))}\qquad\text{for}\;\; i_L(F):=\min\bigl\{0\leq i\leq|\mathrm{nf}(F)| \ :\ d(F^{(i)})\ge L\bigr\}. 
\end{equation*}
 Note that $i_L(F)$ and $\sigma_L(F)$ are well defined by Lemma \ref{lem:d+deletion}. In particular, we may view $\sigma_L$ as a function $\sigma_L:\mathcal{S}_{\geq L}\rightarrow \mathcal{S}$. Also for $L>1$,  define  
\begin{equation}\label{eq:TL}
\mathcal T_L:=\bigl\{G\in \mathcal{S}\ :\ d(G)\ge L>d(G^-)\bigr\},
\end{equation}
which we view as the smallest elements of $\mathcal{S}$ whose $d$-value is at least $L$. We pause now to gather several facts about $\sigma_L$ and $\mathcal{T}_L$. 
\begin{lem}\label{lem:properties+sigma+T}
Let $L>1$ and let $\sigma_L$ and $\mathcal{T}_L$ be as defined above. Then
\begin{enumerate}
\item[\textup{(1)}] $\sigma_L(\mathcal{S}_{\geq L})=\mathcal{T}_L$. \vspace{.1cm}   
\item[\textup{(2)}] $\mathcal{T}_L$ is finite.  
\end{enumerate}
\end{lem}
\begin{proof} 
Let $F\in \mathcal{S}_{\geq L}$ and write $i=i_L(F)$ and $\sigma_L(F)=F^{(i)}=G$. Then $d(G)\geq L$ by definition of $i_L(F)$. On the other hand, we have that $|\textup{nf}(G)|=|s_i(F)|=i$. Moreover, $i\geq1$ since $L>1$ and $d(F^0)=1$ by Lemma \ref{lem:d+deletion}. Hence, Lemma \ref{lem:suffix+deletion} part (4) implies that \[F^{(i-1)}=(F^{(i)})^{(i-1)}=G^{(i-1)}=G^{|\textup{nf}(G)|-1}=G^{-}.\]
But $d(F^{(i-1)})<L$ by definition of $i_L(F)$. Hence, $d(G^{-})<L$ and $G\in \mathcal{T}_L$. In particular, we have shown that $\sigma_L(\mathcal{S}_{\geq L})\subseteq \mathcal{T}_L$. 

Conversely, suppose that $G\in \mathcal{T}_L$ and $m=|\textup{nf}(G)|\geq1$. Then $d(G^{(m)})=d(G)\geq L$ and $d(G^{(i)})\leq d(G^{(m-1)})=d(G^{-})<L$ for all $i\leq m-1$ by Lemma \ref{lem:d+deletion}. Hence, $\sigma_L(G)=G^{(m)}=G$ by the definition of $\sigma_L$. In particular, we deduce that 
\[\mathcal{T}_L=\sigma_L(\mathcal{T}_L)\subseteq \sigma_L(\mathcal{S}_{\geq L})\subseteq \mathcal{T}_L,\]
from which claim (1) follows.

For claim (2), write $d_{\max}=\max\{d(f):f\in \Gamma\}$. Then it follows form Lemma \ref{lem:d+deletion} that 
\[\mathcal{T}_L\subseteq \{G\in \mathcal{S}: d(G)\leq L\cdot d_{\max}\}.\]
On the other hand, if $m=|\textup{nf}(G)|$, then Lemma \ref{lem:d+deletion} implies that $d(G)\geq 2^m$. Combining this with the subset relation above, we see that $|\textup{nf}(G)|\leq \lceil \log_2(L\cdot d_{\max})\rceil$ for all $G\in\mathcal{T}_L$. But $\textup{nf}(\cdot)$ uniquely identifies elements of $\mathcal{S}$, so there are at most $(\#\Gamma)^{\lceil \log_2(L\cdot d_{\max})\rceil}$ elements of $\mathcal{T}_L$.     
\end{proof}
Finally, for $G\in \mathcal{S}$, we define the suffix class,  
\begin{equation*}
\mathcal{C}_G:=\bigl\{F\in \mathcal{S}\ :\ \mathrm{nf}(G)\leq \mathrm{nf}(F)\bigr\},
\end{equation*}
to be the set of $F\in \mathcal{S}$ whose normal form contains the normal form of $G$ as a suffix. In particular, we have the following decomposition.   
\begin{lem}\label{lem:decomp} Let $L>1$ and $G\in\mathcal{T}_L$. Then $\sigma_L^{-1}(G)=\mathcal{C}_G$. Hence, 
\[\mathcal{S}_{\geq L}=\bigsqcup_{G\in\mathcal{T}_L}\mathcal{C}_G\]
is a disjoint partition of $\mathcal{S}_{\geq L}$ into finitely many non-empty components.  
\end{lem}
\begin{proof} 
If $F\in \mathcal{S}_{\geq L}$ is such that $\sigma_L(F)=G$, then $\textup{nf}(G)=\textup{nf}(F^{i_L(F)})=s_{i_L}(F)$ by Lemma \ref{lem:suffix+deletion} part (1), so that the normal form of $G$ is a suffix of the normal form of $F$. Hence, $\sigma_L^{-1}(G)\subseteq \mathcal{C}_G$. 

Conversely, suppose that $F\in \mathcal{C}_G$ and write $m=|\textup{nf}(G)|$. Then since $\textup{nf}(G)\leq \textup{nf}(F)$ and suffixes of a \emph{word} are uniquely determined by their length, we must have that $\textup{nf}(G)=s_m(F)$ so that $G=F^{(m)}$. But then $d(F^{(m)})=d(G)\geq L$ by assumption on $G$, so that $m\geq i_{L}(F)$ by definition of $i_L(F)$. On the other hand, for $0\leq i\leq m-1$ we have that $F^{(i)}=(F^{(m)})^{(i)}=G^{(i)}$ by Lemma \ref{lem:suffix+deletion} part (4). Hence, 
\[d(F^{(i)})=d(G^{(i)})\leq d(G^{(m-1)})=d(G^{-})<L\]
by Lemma \ref{lem:d+deletion} and our assumption that $G\in\mathcal{T}_L$. Therefore, $m=i_L(F)$, $G=\sigma_L(F)$, and $\mathcal{C}_G\subseteq \sigma_L^{-1}(G)$ as claimed.

Finally, we explain the decomposition of $\mathcal{S}_{\geq L}$ given in Lemma \ref{lem:decomp}. Since $\sigma_L:\mathcal{S}_{\geq L}\rightarrow \mathcal{T}_L$ is a well-defined function with image $\mathcal{T}_L$ by Lemma \ref{lem:properties+sigma+T}, its fibers are pairwise disjoint and cover $\mathcal{S}_{\geq L}$. Moreover, the fibers are exactly $\mathcal{C}_G$ for $G\in\mathcal{T}_L$ by the first claim of this Lemma. Finally, $\mathcal{T}_L$ is finite by Lemma \ref{lem:properties+sigma+T}, and $\mathcal{C}_G$ is non-empty since $\textup{nf}(G)\leq \textup{nf}(G)$.      
\end{proof}
\section{From Degrees to Heights}\label{sec:d+to+h} 
In this section, we use normal forms and sufficient information about the growth rate of a multiplicaitve function $d$ (e.g., the degree of a finite map) to deduce a growth rate of a related nonnegative function $h$. In particular, we make the argument outlined in subsection \ref{subsec:sketch} precise. Specifically, we use a suffix-closed normal-form to give a (type of) right coset decomposition of $\mathcal{S}$ into finitely many suffix classes $\mathcal{C}_{G}$; so we use $\mathcal{C}_G$ in place of $\mathcal{S}_G$ in the sketch in subsection \ref{subsec:sketch}. Moreover, we then assume that \emph{$d$ grows regularly on suffix classes}, i.e., the shape of the asymptotic count of 
\begin{equation}\label{eq:class+count+regularity}
\#\{F\in\mathcal{C}_G: d(F)\leq x\}
\end{equation}
does not depend on $G\in\mathcal{S}$. Finally, we tie the degree counts on \eqref{eq:class+count+regularity} for $G\in\mathcal{S}$ and the height counts $\#\{F\in\mathcal{C}_G: h(F)\leq x\}$ together using a Tate telescoping argument.           
\subsection{Relating the Growth of a Pair of Functions}
We are ready to prove the main theorem of this section, which relates the growth rate of a multiplicative function $d$ on a semigroup to the growth rate of a suitable nonnegative companion function $h$. The reader is encouraged to have the examples $d(F):=\deg(F)$ and $h(F):=h_\mathcal{L}(F(P))$ in mind. 
\begin{thm}\label{thm:d+to+h} Let $(\mathcal{S},\Gamma,\mathcal{N})$ be a finitely generated semigroup with a suffix-closed normal form, and suppose that $\mathcal{S}$ is equipped with two functions, $d:\mathcal{S}\rightarrow\mathbb{Z}_{\geq2}$ and $h: \mathcal{S}\rightarrow \mathbb{R}_{\geq0}$, with $d$ multiplicative. Furthermore, assume the following properties hold:
\begin{enumerate}
    \item[\textup{(1)}] There is a constant $c_1\geq0$ such that 
    \[\bigg|\frac{h(FG)}{d(FG)}-\frac{h(G)}{d(G)}\bigg|\leq \frac{c_1}{d(G)}\]
    for all $F,G\in \mathcal{S}$.\vspace{.2cm}  
    \item[\textup{(2)}] There are constants $c_2,c_2',L_0>0$ such that \[c_2 d(F)\leq h(F)\leq c_2' d(F)\] 
    for all $F\in \mathcal{S}_{\geq L_0}$. \vspace{.2cm}
    \item[\textup{(3)}] There is a constant $\rho>0$ and a constant $c_G>0$ for all $G\in \mathcal{S}$ such that 
    \[\#\{F\in \mathcal{C}_G\,: d(F)\leq x\}= c_Gx^{\rho}+o(x^\rho).
    \vspace{.2cm} 
    \]
\end{enumerate}
Then there is a constant $c_3>0$ such that 
\[\#\{F\in \mathcal{S}\,: h(F)\leq x\}=c_3x^{\rho}+o(x^\rho),\]
and $c_3$ is given explicitly by 
\[c_3=\lim_{L\rightarrow\infty}\sum_{G\in\mathcal{T}_L}c_G\Big(\frac{h(G)}{d(G)}\Big)^{-\rho}.\]
\end{thm}
\begin{proof}[(Proof of Theorem \ref{thm:d+to+h})]
Define the function $\psi:\mathcal{S}\rightarrow\mathbb{R}_{\geq0}$ given by $\psi(F)=h(F)/d(F)$. We first note that $\psi$ is nearly constant on classes. To see this, we assume throughout that  
\begin{equation}\label{eq:L+bd}
L>\max\{L_0,c_1/c_2\}.
\end{equation} 
Pick $G\in\mathcal{T}_L$ and suppose that $F\in\mathcal{C}_G$. Then $F=F'G$ for some $F'\in \mathcal{S}^1$ by Lemma \ref{lem:suffix+deletion} so assumption (1) implies that $|\psi(F)-\psi(G)|\leq c_1/d(G)$. Dividing by $\psi(G)$ we see that 
\begin{equation}\label{eq:epsilon+L}
\Big|\frac{\psi(F)}{\psi(G)}-1\Big|\leq\frac{c_1}{d(G)}\frac{d(G)}{h(G)}\leq \frac{c_1}{c_2L}:=\epsilon_L\rightarrow 0
\end{equation}
as $L\rightarrow\infty$. Here we use that $h(G)\geq c_2d(G)\geq c_2 L$ by assumption (2) and the fact that $G\in\mathcal{T}_L$ for some $L>L_0$. Note also that $\epsilon_L<1$ since $L\geq c_1/c_2$ by \eqref{eq:L+bd}. On the other hand, for all $F$ we have that $h(F)\leq x$ if and only if $d(F)\leq x/\psi(F)$. Moreover, for $F\in\mathcal{C}_G$ and $G\in\mathcal{T}_L$ we have from \eqref{eq:epsilon+L} that 
\[(1-\epsilon_L)\psi(G)\leq \psi(F)\leq (1+\epsilon_L)\psi(G).\]
Hence, we see that \vspace{.15cm} 
\[
\scalemath{.9}{
\#\bigg\{F\in\mathcal{C}_G\,:\, d(F)\leq\frac{x}{\psi(G)(1+\epsilon_L)}\bigg\}\leq \#\{F\in\mathcal{C}_G\,:\, h(F)\leq x\}\leq 
\#\bigg\{F\in\mathcal{C}_G\,:\, d(F)\leq\frac{x}{\psi(G)(1-\epsilon_L)}\bigg\}}
\vspace{.15cm} 
.\]
In particular, we deduce that
\vspace{.1cm}
\begin{equation}\label{eq:class+setimate}
\scalemath{.93}{
c_G(\psi(G))^{-\rho}(1+\epsilon_L)^{-\rho}+o(1)\leq \frac{\#\{F\in\mathcal{C}_G\,:\, h(F)\leq x\}}{x^\rho}\leq c_G(\psi(G))^{-\rho}(1-\epsilon_L)^{-\rho}+o(1)}
\vspace{.1cm} 
\end{equation}
as $x\rightarrow\infty$ for all $G\in\mathcal{T}_L$ by assumption (3). From here, define
\[M(x)=\#\{F\in \mathcal{S}\,:\, h(F)\leq x\}\;\;\;\;\text{and}\;\;\;\; 
\kappa_{L}=\#\{F\in \mathcal{S}\,:\, d(F)\leq L\}<\infty.
\]
Note that $\kappa_{L}$ is finite, since $d(F)\leq L$ implies that $2^{|\textup{nf}(F)|}\leq d(F)\leq L$ by Lemma \ref{lem:d+deletion}, so there are at most $\#\Gamma^{\log_2 (L)}$ such elements. From here, the decomposition in Lemma \ref{lem:decomp} implies that \vspace{.1cm} 
\begin{equation}\label{eq:almost}
\Big| M(x)-\sum_{G\in \mathcal{T}_L}\#\{F\in\mathcal{C}_G\,:\, h(F)\leq x\}\Big|\leq\kappa_{L}. 
\end{equation}
Now define the constant 
\[\beta_L:=\sum_{G\in \mathcal{T}_L}c_G\psi(G)^{-\rho}.\]
Then combining \eqref{eq:class+setimate} and \eqref{eq:almost} we deduce that 
\begin{equation}\label{eq:liminf+limsup}
\beta_L(1+\epsilon_L)^{-\rho}\leq \liminf_{x\rightarrow\infty}\frac{M(x)}{x^\rho}\leq \limsup_{x\rightarrow\infty}\frac{M(x)}{x^\rho}\leq \beta_L(1-\epsilon_L)^{-\rho}.  
\end{equation}
In particular, we see from \eqref{eq:liminf+limsup} that
\vspace{.1cm} 
\[
\scalemath{.9}{
\liminf_{x\rightarrow\infty}\frac{M(x)}{x^\rho}\leq \limsup_{x\rightarrow\infty}\frac{M(x)}{x^\rho}\leq \beta_L(1-\epsilon_L)^{-\rho}=\Big(\frac{1-\epsilon_L}{1+\epsilon_L}\Big)^{-\rho}\beta_L(1+\epsilon_L)^{-\rho}\leq 
\Big(\frac{1-\epsilon_L}{1+\epsilon_L}\Big)^{-\rho}
\liminf_{x\rightarrow\infty}\frac{M(x)}{x^\rho}}
\vspace{.1cm}
\]
holds for all $L$ satisfying \eqref{eq:L+bd}. Hence, letting $L\rightarrow\infty$ so that $\epsilon_L\rightarrow0$, we deduce that 
\[
\liminf_{x\rightarrow\infty}\frac{M(x)}{x^\rho}=\limsup_{x\rightarrow\infty}\frac{M(x)}{x^\rho}=\lim_{x\rightarrow\infty}\frac{M(x)}{x^\rho}.  
\]
Call $c_3=\lim M(x)/x^{\rho}$ as $x\rightarrow\infty$. Then substituting $c_3$ into \eqref{eq:liminf+limsup}, we have that 
\[(1+\epsilon_L)^{-\rho}\leq \frac{c_3}{\beta_L}\leq (1-\epsilon_L)^{-\rho}\]
for all $L$ satisfying \eqref{eq:L+bd}. Hence. letting $L\rightarrow\infty$ we deduce that $c_3=\displaystyle{\lim_{L\rightarrow\infty}\beta_L}$ as claimed. 

Finally, we show that $c_3>0$. To see this assume that $c_3=0$. Then assumptions (2) and (3), our asymptotic for $M(x)$ already established, and Lemma \ref{lem:decomp} together imply that
\vspace{.1cm} 
\begin{equation*}\begin{split}
\bigg(\sum_{G\in\mathcal{T}_{L_0}}c_G\bigg)(c_2')^{-\rho}x^{\rho}+o(x^{\rho})&=\sum_{G\in\mathcal{T}_{L_0}}\#
\big\{F\in\mathcal{C}_G\,:d(F)\leq x/(c2')\big\} \\[5pt]
&=\#\{F\in \mathcal{S}_{\geq L_0}\,:d(F)\leq x/(c2')\}\\[5pt]
&\leq \#\{F\in \mathcal{S}\,:d(F)\leq x/(c2')\}\\[5pt] 
&=\#\{F\in \mathcal{S}\,:h(F)\leq x\}\\[5pt] 
&=o(x^{\rho})
,
\vspace{.1cm} 
\end{split} 
\end{equation*}
and we obtain a contradiction. Here we use that $c_G,c_2'>0$ by assumption.  
\end{proof}
\begin{remark} It is worth pointing out that the constant  $c_4:=\sum_{G\in\mathcal{T}_L}c_G$ gives the constant multiple of $x^{\rho}$ in the asymptotic growth of $d$ on $\mathcal{S}$, independent of $L$. To see this, define
\vspace{.1cm} 
\[N(x)=\#\{F\in \mathcal{S}: d(F)\leq x\}\;\;\;\;\text{and}\;\;\;\; \kappa_{L}=\#\{F\in \mathcal{S}\,:\, d(F)\leq L\}<\infty.\] 
Indeed, $\kappa_{L}$ is finite since $d(F)\leq L$ implies $2^{|\textup{nf}(F)|}\leq L$ by Lemma \ref{lem:d+deletion}, and there are at most $\#\Gamma^{\log_2(L)}$ such $F$'s. Likewise, the decomposition in Lemma \ref{lem:decomp} implies that
\[\Big| N(x)-\sum_{G\in \mathcal{T}_L}\#\{F\in\mathcal{C}_G\,:\, d(F)\leq x\}\Big|\leq\kappa_{L}.\]
Hence, assumptions (3) implies that 
\[c_4x^\rho+o(x^{\rho})=\Big(\sum_{G\in\mathcal{T}_L}c_G\Big)x^{\rho}+o(x^\rho)=N(x),\]
since $\mathcal{T}_L$ is finite. Therefore, $N(x)\sim c_4x^{\rho}$ is claimed.    
\end{remark}
\begin{remark}\label{rem:probability} In particular, the measure $\nu_L$ on $\mathcal{T}_L$ induced by $\nu_L(G):=c_G/c_4$ gives a probability measure on $\mathcal{T}_L$, and we can view the coefficient $c_3$ from this point of view, 
\[c_3=c_4\lim_{L\rightarrow\infty}\mathbb{E}_{\nu_L}[\psi^{-\rho}].\]
Moreover, the dominated convergence theorem suggests that this limit is the expected value of $\psi^{-\rho}$, some type of canonical height function on an infinite object raised to the $-\rho$ on some limiting probability space. We plan to study this perspective in future work. 
\end{remark}
\subsection{Heights in Polarized Semigroups}
Throughout this subsection, we assume that $\mathcal{S}$ is a polarized semigroup of morphisms on a variety $V$, i.e., there exists an  ample invertible sheaf $\mathcal{L}$ of $V$ and integers $d(F)\geq2$ for all $F\in \mathcal{S}$ such that $F^*(\mathcal{L}) \cong \mathcal{L}^{\otimes d(F)}$. In particular, the multiplicative map $d$ on $\mathcal{S}$ we refer to in this section is given by this relation. Moreover, we assume that $\mathcal{L}$ and all maps in $\mathcal{S}$ are defined over a number field $K$ and associate to $\mathcal{L}$ a Weil height function $h_{\mathcal{L}}:V(\overline{K})\rightarrow\mathbb{R}_{\geq0}$. Then, of fundamental importance, we have the following generalization of Tate's telescoping lemma; see \cite[Lemma 2.1]{MR3958050} for a proof.   
\begin{lem}\label{lem:tate} Let $\mathcal{S}$, $\mathcal{L}$, $d$, and $h_{\mathcal{L}}$ be as above. Then there is a $B=B(\mathcal{S},\mathcal{L},d,h)\geq0$ such that 
\[\Big|\frac{h_{\mathcal{L}}(F(Q))}{d(F)}-h_{\mathcal{L}}(Q)\Big|\leq B\]
for all $F\in \mathcal{S}$ and all $Q\in V(\overline{K})$.  
\end{lem}
Next, we note the following fact from \cite[Lemma 2.11]{MR4988033}, which will be used to reduce the height counting problem for non-preperiodic points to the height counting problem for points of sufficiently large height; recall that $P$ is preperiodic if there exist $g\in\mathcal{S}$ and $f\in \mathcal{S}^1$ such that $g\circ f(P)=f(P)$.  
\begin{lem}\label{lem:preperiodic} Let $\mathcal{S}$, $\mathcal{L}$, $d$, and $h_{\mathcal{L}}$ be as above. Then the following statements hold:
\vspace{.1cm} 
\begin{enumerate}
    \item[\textup{(1)}] If $P$ is a preperiodic point for $\mathcal{S}$, then there exists a constant $C$ such that $h_{\mathcal{L}}(P)\leq C$. \vspace{-.2cm}    
    \item[\textup{(2)}] Assume $\mathcal{S}$ is finitely generated. If $P$ is a not a preperiodic point for $\mathcal{S}$ and $C'\geq0$, then there exists $L=L(P,C')$ such that $h_{\mathcal{L}}(G(P))>C'$ for all $G\in \mathcal{S}$ with $d(G)\geq L$.   
\end{enumerate}
\end{lem}
\begin{proof} Only statement (2) differs from the original statements in \cite[Lemma 2.11]{MR4988033}. Assume that $\mathcal{S}=\langle f_1,\dots,f_r\rangle$ and let $d_{\max}:=\max\{d(f_1),\dots,d(f_r)\}$. Let $N=N(V,\mathcal{S},P,C')$ be the constant in Lemma \cite[Lemma 2.11]{MR4988033} and assume that $G\in \mathcal{S}$ satisfies $d(G)\geq d_{\max}^N$. Then it is clear that $d_{\max}^{|G|}\geq d(G)$ so that $|G|\geq \log_{d_{\max}}(d(G))$. But then we see that 
\[|G|\geq \log_{d_{\max}}(d(G))\geq\log_{d_{\max}}(d_{\max}^N)\geq N.\]
Hence, \cite[Lemma 2.11]{MR4988033} implies that $h_{\mathcal{L}}(G(P))> C'$. In particular, the constant $L:=d_{\max}^N$ has the desired property.  
\end{proof}

In particular, if we view $P$ as fixed and define $h:\mathcal{S}\rightarrow\mathbb{R}_{\geq0}$ by $h(F):=h_{\mathcal{L}}(F(P))$, then the pair of functions $(d,h)$ on $\mathcal{S}$ satisfy the hypotheses (1) and (2) of Theorem \ref{thm:d+to+h} for all non-preperiodic basepoints. 
\begin{lem}\label{lem:(d,h)+preperiodic}
 Let $\mathcal{S}$, $\mathcal{L}$, $d$, and $h_{\mathcal{L}}$ be as above and assume that $(\mathcal{S},\Gamma,\mathcal{N})$ is a suffix-closed normal form on $\mathcal{S}$. Moreover, assume that $P\in V(\overline{K})$ is a not a preperiodic point for $\mathcal{S}$. Define the function $h: \mathcal{S}\rightarrow\mathbb{R}_{\geq0}$ given by $h(F):=h_{\mathcal{L}}(F(P))$. Then the pair $(d,h)$ satisfy the hypotheses \textup{(1)} and \textup{(2)} of Theorem \ref{thm:d+to+h}. 
\end{lem}
\begin{proof}
Let $B$ be the constant from Lemma \ref{lem:tate}. For hypothesis (1) of Theorem \ref{thm:d+to+h}, set $Q=G(P)$ and multiply the bound in Lemma \ref{lem:tate} by $1/d(G)$ so that $c_1:=B$ has the desired property. 

For hypothesis (2) of Theorem \ref{thm:d+to+h}, set $c_2':=h_{\mathcal{L}}(P)+B$ so that Lemma \ref{lem:tate} applied to the point $P=Q$ gives 
\[h(F)\leq d(F)(h_{\mathcal{L}}(P)+B)=c_2' d(F),\]
with no assumptions on $P$ or $F$. For the lower bound, assume that $P$ is not a preperiodic point for $\mathcal{S}$ and choose the constant $L=L(P,B)$ associated to $P$ and $B$ from Lemma \ref{lem:preperiodic}. Define 
\[c_2:=\min_{G\in\mathcal{T}_{L}}\Big\{\frac{1}{d(G)}\Big(h_{\mathcal{L}}\big(G(P))-B\Big)\Big\}.\]
Then we note that $c_2>0$: indeed $h_{\mathcal{L}}\big(G(P))-B>0$ for all $G\in\mathcal{T}_L$ by definition of $L$ and the fact that $d(G)\geq L$ by definition of $\mathcal{T}_L$; see \eqref{eq:TL}. Moreover, $\mathcal{T}_L$ is finite by Lemma \ref{lem:properties+sigma+T}. Hence, $c_2$ is the minimum of finitely many strictly positive numbers and so strictly positive. Now suppose that $d(F)\geq L$. Then Lemma \ref{lem:suffix+deletion} and Lemma \ref{lem:decomp} imply that $F=W\circ G$ for some $W\in \mathcal{S}^1$ and $G\in\mathcal{T}_L$. But then Lemma \ref{lem:tate}, with $W$ in place of $F$ and $G(P)$ in place of $Q$, implies that 
\[
\scalemath{.95}{
h(F)=h_{\mathcal{L}}(F(P))=h_{\mathcal{L}}(W(G(P)))\geq d(W)\Big(h_{\mathcal{L}}\big(G(P))-B\Big)=\frac{d(F)}{d(G)}\Big(h_{\mathcal{L}}\big(G(P))-B\Big)\geq c_2d(F)}.   
\]
Thus, $L_0=L(P,B)$ and $c_2$ satisfy the desired conditions for hypothesis (2) of Theorem \ref{thm:d+to+h}. 
\end{proof}
\section{Counting Weighted Walks on a Directed Multigraph}\label{sec:counting+walks}
In particular, to apply Theorem \ref{thm:d+to+h} to heights, it remains to study the growth rate of $d$ in $\mathcal{C}_G$ subject to a normal form. Specifically, we still need assumption (3) of Theorem \ref{thm:d+to+h} to hold. As we will see in Section \ref{sec:trace}, when $\mathcal{S}$ is partially commutative semigroups with the Cartier-Foata normal form \cite{MR239978}, this problem (of counting $d$ on $\mathcal{C}_G$) is equivalent to counting walks on an associated directed graph; see Proposition \ref{prop:CF+implies+suffix+closed+directed+graph+walks} below. 

With this in mind, we derive Dirichlet series generating functions for the number of walks of a given multiplicative weight in a directed multigraph in this section. This process is likely well known to the experts in graph theory. Nevertheless, without an exact reference, we have chosen to derive the asymptotics we need from scratch. In particular, we first translate the problem of counting walks on a graph into a matrix generating function problem via the Transfer Matrix Method \cite[\S V.6]{flajolet2009analytic}. We then study the singularities of this generating function via the Perron--Frobenius Theorem \cite[\S8]{horn2012matrix}. Finally, we extract asymptotics from singularities via the Wiener--Ikehara Theorem \cite[\S8.3]{montgomery2007multiplicative}.

\subsection{Generating Functions}
Let $\mathcal{D}=(\mathcal{V},\mathcal{E}, \mu)$ be a finite directed multigraph with vertex set $\mathcal{V}=\{v_1,\dots, v_n\}$, edge set $\mathcal{E}$, and a mapping $\mu:\mathcal{E}\rightarrow\mathcal{V}\times\mathcal{V}$. In particular, we allow multiple edges between vertices and loops (in this section), and we use $\mu$ to indicate the initial and terminal vertex of an edge. Moreover, for $\epsilon\in\mathcal{E}$ we write $\epsilon=a\rightarrow b$ in place of $\mu(\epsilon)=(a,b)\in \mathcal{V}\times \mathcal{V}$ to ease the notation. Since we assume that $\mathcal{E}$ is also finite, we can define the \emph{transfer matrix} $A=(A_{ij})\in \textup{Mat}_{n\times n}(\mathbb{Z}_{\geq0})$ given by
\[A_{ij}:=\#\{\epsilon\in\mathcal{E}:\epsilon=v_i\rightarrow v_j\}.\]
Furthermore, to each vertex $v_i\in\mathcal{V}$ we assign a weight $d(v_i)>1$. Later, we will assume that the weights are integers but this restriction is not necessary at this point.

Next, we call a pair of sequences $\gamma=\big((a_0,\dots, a_{m}),(\epsilon_1,\dots,\epsilon_{m})\big)$ for some $a_i\in\mathcal{V}$, some $\epsilon_j\in\mathcal{E}$, and some $m\geq1$ a \emph{walk} in $\mathcal{D}$ of length $|\gamma|=m$ if $\epsilon_i=a_{i-1}\rightarrow a_{i}$ for all $1\leq i\leq m$. Likewise, we view a single vertex as a walk of length $0$, i.e., $\gamma=((a_0),\varnothing)$. Moreover, we extend $d$ to walks in $\mathcal{D}$ multiplicatively by defining the weight of $\gamma$ to be
\[d(\gamma)=d(a_0)\cdots d(a_m).\]

In particular, it is our goal in this section to find a generating function for walks of a given length and use this generating function to count such walks. More generally, with suffix classes and partially commutative semigroups in mind, we want to count walks with a given terminal segment. To make this precise, we fix a bit more notation. Specifically, given a walk $\gamma$ with vertex sequence $(a_0,\dots, a_m)$ in $\mathcal{D}$ we write $\iota(\gamma)=a_0$ and $\tau(\gamma)=a_m$ to be the initial and terminal vertices of $\gamma$ respectively. In particular, if $\gamma_1$ and $\gamma_2$ are walks in $\mathcal{D}$ and $\epsilon=\tau(\gamma_1)\rightarrow\iota(\gamma_2)$ is an edge of $\mathcal{D}$, then we define the \emph{concatenation} $\gamma_1\cdot_{\epsilon} \gamma_2$ to be the walk obtained by concatenating the sequences of vertices and adding the edge $\epsilon$ in the middle to join the walks: for $\gamma_1=\big((a_0,\dots, a_m),(\epsilon_1,\dots,\epsilon_{m})\big)$ and $\gamma_2=\big((a_0',\dots, a_n'),(\epsilon_1',\dots,\epsilon_{n}')\big)$, we set
\[\gamma_1\cdot_{\epsilon}\gamma_2= \big((a_0,\dots, a_m,a_0',\dots,a_n'),(\epsilon_1,\dots,\epsilon_{m},\epsilon,\epsilon_1',\dots,\epsilon_{n}')\big).\]
Moreover, it is clear that $d(\gamma_1\cdot_{\epsilon}\gamma_2)=d(\gamma_1)\cdot d(\gamma_2)$ and that $|\gamma_1\cdot_\epsilon\gamma_2|=|\gamma_1|+|\gamma_2|+1$.

From here, we form the generating function for walks with a given terminal segment. To do this, we fix a walk $\eta$ of length $r\geq0$ on $\mathcal{D}$ and say that
$\eta$ is a \emph{terminal segment} of a walk $\gamma$ if the last $r+1$ vertices and $r$ edges of $\gamma$ are those of $\eta$; equivalently, either $\gamma=\eta$ or we may write $\gamma=\omega \cdot_\epsilon \eta$ for some walk $\omega$ and some edge $\epsilon=\tau(\omega)\rightarrow\iota(\eta)$. Then put
\[\mathcal{W}_\eta=\big\{\gamma\,:\,\text{$\eta$ is a terminal segment of $\gamma$}\big\}\]
to be the set of walks that end with the walk $\eta$. Likewise for $k>0$, set
\[a_k=a_k(d,\eta):=\#\{\gamma\in\mathcal{W}_\eta\,:\,d(\gamma)=k\}\;\;\;\;\;\text{and}\;\;\;\;\; Z_\eta(s)=\sum_{\gamma\in\mathcal{W}_\eta}d(\gamma)^{-s}.
\]
Finally, we define the \emph{weighted transfer matrix}: for $s\in\mathbb{C}$ let $M(s)\in\textup{Mat}_{n\times n}(\mathbb{C})$ and the row vector $\delta(s)$ be given by
\begin{equation}\label{eq:M+delta}
    M_{ij}(s)=A_{ij}d(v_j)^{-s}\;\;\;\text{and}\;\;\; \delta(s)=\big(d(v_1)^{-s},\dots, d(v_n)^{-s}\big).
\end{equation}
In particular, we note the following fact about the weighted transfer matrix, which generalizes the standard fact that the entry $(A^{m})_{ij}$ gives the number of walks of length $m$ that start at $v_i$ and end at $v_j$; see, for instance, \cite[Theorem 4.7.1]{StanleyEC1}.
\begin{lem}\label{lem:transfer+matrix} For all $m\geq0$ and $1\leq j\leq n$, we have that
\[\sum_{\substack{|\gamma|=m\\ \tau(\gamma)=v_j}}d(\gamma)^{-s}=\Big(\delta(s)M(s)^{m}\Big)_j.\]
\end{lem}
\begin{proof} Suppose $m=0$. Then the only walk of length $0$ ending at $v_j$ is the single vertex $v_j$. Hence, the claim in the lemma is simply that $d(v_j)^{-s}=(\delta(s)M(s)^{0})_j=\delta(s)_j=d(v_j)^{-s}$ for all $j$, which is trivially true.

From here we proceed by induction on $m$. Suppose that $m\geq1$. Then every walk of length $m$ ending at $v_j$ can be uniquely written as the concatenation of a walk of length $m-1$, ending at some $v_i$, followed by one of the $A_{ij}$ edges $v_i\rightarrow v_j$. That is, we obtain a bijection of sets:
\[\big\{\gamma: |\gamma|=m\;\text{and}\; \tau(\gamma)=v_j\big\}\longleftrightarrow\bigsqcup_{i=1}^{n} \big\{\omega: |\omega|=m-1\;\text{and}\; \tau(\omega)=v_i\big\}\times\big\{\epsilon\in\mathcal{E}:\epsilon=v_i\rightarrow v_j\big\}.\]
Moreover, under this bijection we have that $d(\gamma)^{-s}=d(\omega)^{-s}d(v_j)^{-s}$, and $A_{ij}=\#\{\epsilon\in\mathcal{E}:\epsilon=v_i\rightarrow v_j\}$ by definition. Hence,
\begin{equation*}
\sum_{\substack{|\gamma|=m\\ \tau(\gamma)=v_j}}d(\gamma)^{-s}
=\sum_{i=1}^n \bigg(\sum_{\substack{|\omega|=m-1\\ \tau(\omega)=v_i}}d(\omega)^{-s}\bigg)A_{ij}d(v_j)^{-s}=\sum_{i=1}^n\Big(\delta(s)M(s)^{m-1}\Big)_{\hspace{-.05cm}i}\,M_{ij}(s)
\end{equation*}
by the induction hypothesis. However, the right-hand side is $\big(\delta(s)M(s)^{m}\big)_j$ as desired.
\end{proof}
Next, we make a few simple observations about the convergence of the Dirichlet series generating functions defined above.
\begin{lem}\label{lem:basic+convergence} Let $d_{\min}=\min_i\{d(v_i)\}>1$ and $\|A\|_{\infty}:=\max_i\{\sum_{j=1}^n A_{ij}\}$ denote the minimum weight and maximum out-degree of vertices of $\mathcal{D}$, respectively. Then each $a_k$ is finite and
\[\sum_{\textup{walks $\gamma$}}|d(\gamma)^{-s}|<\infty\]
for all $\Re(s)>\sigma_0:=\log(\|A\|_{\infty})/\log(d_{\min})$. Here, the sum is over all walks of $\mathcal{D}$. In particular, $Z_\eta(s)$ converges absolutely for all $\Re(s)>\sigma_0$ and all walks $\eta$.
\end{lem}
\begin{proof} A walk $\gamma$ of length $m$ has $m+1$ vertices so satisfies $d(\gamma)\geq d_{\min}^{m+1}$. Thus $d(\gamma)\leq x$ forces $m\leq \log(x)/\log(d_{\min})-1$. Moreover, the number of walks of length $m$ is at most $n (\|A\|_{\infty})^{m}$: from each starting vertex there are at most $\|A\|_{\infty}$ choices at each of the $m$ steps. In particular, there are only finitely many walks of any given weight, so $a_k$ is finite. Likewise, for $\Re(s)=\sigma$ we have that
\[\sum_{\textup{walks $\gamma$}}|d(\gamma)^{-s}|=\sum_{m\geq0}\sum_{|\gamma|=m}d(\gamma)^{-\sigma}\leq \sum_{m\geq0} n (\|A\|_{\infty})^{m}d_{\min}^{-(m+1)\sigma}=(n\,d_{\min}^{-\sigma})\sum_{m\geq 0}\big(\|A\|_{\infty}\,d_{\min}^{-\sigma}\big)^m,\]
which converges whenever $\|A\|_{\infty}\,d_{\min}^{-\sigma}<1$, i.e., when $\sigma>\sigma_0$ as claimed.
\end{proof}
In particular, we put the results above together to establish a meromorphic continuation for the Dirichlet series generating functions for walks with a given terminal segment.
\begin{thm}\label{thm:walk+with+terminal+segment} Let $\mathcal{D}=(\mathcal{V},\mathcal{E},\mu)$ be a finite directed multigraph, let $d:\mathcal{V}\rightarrow \mathbb{Z}_{\geq2}$, and extend $d$ multiplicatively to walks on $\mathcal{D}$. Furthermore, let $M(s)$ and $\delta(s)$ be as in \eqref{eq:M+delta}, let $\sigma_0$ be as in Lemma \ref{lem:basic+convergence}, and let $\eta$ be a fixed walk in $\mathcal{D}$ with $\iota(\eta)=v_q\in\mathcal{V}$. Then for all $\Re(s)>\sigma_0$, the matrix $I-M(s)$ is invertible and
\[Z_\eta(s)=\sum_{k\geq1}\frac{a_k}{k^s}=d(\eta)^{-s}\Big(1+\delta(s)\big(I-M(s)\big)^{-1}\nu(\eta)\Big),\]
where $\nu(\eta)^{\textup{T}}=(A_{1q},\dots,A_{nq})$. In particular, by Cramer's rule we see that
\[Z_{\eta}(s)=d(\eta)^{-s}\Big(1+\frac{\delta(s)\,\textup{adj}\big(I-M(s)\big)\,\nu(\eta)}{\det\big(I-M(s)\big)}\Big)\]
is a rational function in $d(v_1)^{-s},\dots,d(v_n)^{-s}$. Hence, $Z_\eta$ extends to a meromorphic function on $\mathbb{C}$ with poles inside the set $\big\{s\in\mathbb{C}\,: \det\big(I-M(s)\big)=0\big\}$.
\end{thm}
\begin{proof}
 For $\Re(s)=\sigma>\sigma_0$ the entry bound $|M_{ij}(s)|\leq A_{ij}d_{\min}^{-\sigma}$ gives
 \[\|M(s)\|_{\infty}\leq \|A\|_{\infty}\,d_{\min}^{-\sigma}<1,\]
 by assumption on $\sigma$. Moreover, $\|\cdot\|_{\infty}$ is submultiplicative, so that $\|M(s)^k\|_{\infty}\leq \|M(s)\|_{\infty}^k$ and $\sum_{k\geq0}M(s)^k$ converges by direct comparison with a geometric series. In particular, the standard telescoping argument with geometric series identifies the matrix $\sum_{k\geq0}M(s)^k$ as the inverse of $I-M(s)$.

Similar to the proof of Lemma \ref{lem:transfer+matrix}, a walk $\gamma\in \mathcal{W}_\eta\setminus\{\eta\}$ can be uniquely written $\gamma=\omega\cdot_{\epsilon}\eta$ as a concatenation of a walk $\omega$ of length $|\gamma|-|\eta|-1$ with an edge $\epsilon=\tau(\omega)\rightarrow \iota(\eta)$. That is, the association $\gamma\mapsto(\omega,\epsilon)$ between the sets
\[\mathcal{W}_\eta\setminus\{\eta\} \longleftrightarrow\bigsqcup_{j=1}^{n} \big\{\omega:\; \tau(\omega)=v_j\big\}\times\big\{\epsilon\in\mathcal{E}:\epsilon=v_j\rightarrow v_q\big\}\]
is a bijection. Summing $d(\gamma)^{-s}=d(\eta)^{-s}d(\omega)^{-s}$ over this bijection gives
\begin{equation}\label{eq:Z1}
Z_\eta(s)=d(\eta)^{-s}\Big(1+\sum_{j=1}^n\Big(\sum_{\tau(\omega)=v_j}d(\omega)^{-s}\Big)A_{jq}\Big).
\end{equation}
On the other hand, Lemma \ref{lem:transfer+matrix} implies that
\[\sum_{\tau(\omega)=v_j}d(\omega)^{-s}=\sum_{m\geq0} \sum_{\substack{|\omega|=m\\ \tau(\omega)=v_j}}d(\omega)^{-s}
=\sum_{m\geq 0}\Big(\delta(s)M(s)^{m}\Big)_j.
\]
In particular, substituting this expression into \eqref{eq:Z1} we deduce that
\begin{equation}\label{eq:Z+eta+expression}
\begin{split}
Z_\eta(s)&=d(\eta)^{-s}\bigg(1+\sum_{j=1}^n
\sum_{m\geq 0}\Big(\delta(s)M(s)^{m}\Big)_jA_{jq}\bigg)\\[5pt]
&=d(\eta)^{-s}\Big(1+\sum_{m\geq0}
\delta(s)M(s)^{m}\nu(\eta)\Big)\\[5pt]
&=d(\eta)^{-s}\Big(1+\delta(s)\Big(\sum_{m\geq0}
M(s)^{m}\Big)\nu(\eta)\Big)\\[5pt]
&=d(\eta)^{-s}\Big(1+\delta(s)\big(I-M(s)\big)^{-1}\nu(\eta)\Big)
\end{split}
\end{equation}
as desired. Here, our reordering of terms is justified since the sums above converge absolutely. The rest of the claim follows.
\end{proof}
\subsection{Poles and Asymptotics}
Next, in order to obtain asymptotics from generating functions, we study the poles of $Z_\eta$. In particular, we make use of the Perron--Frobenius theorem for primitive matrices; recall that a square matrix $A$ with nonnegative entries is \emph{primitive} if there exists $k>0$ such that $A^k$ has strictly positive entries (written $A^k>0$ for short). Throughout, $\textup{sr}(A)$ denotes the spectral radius of a square matrix $A$.
\begin{thm}[Perron--Frobenius \cite{seneta1981non}]
\label{thm:P-F}
Let $A$ be a nonnegative, primitive matrix. Then there exists a real number $\lambda>0$ \textup{(}namely, $\lambda=\textup{sr}(A)$\textup{)} such that:
\begin{enumerate}
\item[\textup{(1)}] $\lambda$ is an eigenvalue of $A$. \vspace{.15cm}
\item[\textup{(2)}] If $\alpha\neq\lambda $ is an eigenvalue of $A$, then $|\alpha|<\lambda$.
\vspace{.15cm}
\item[\textup{(3)}] $\lambda$ is a simple root of the characteristic polynomial of $A$.
\vspace{.15cm}
\item[\textup{(4)}] If $B$ is a square matrix such that $0\leq B_{ij}\leq A_{ij}$ for all corresponding entries of $A$ and $B$ and if $\beta$ is an eigenvalue of $B$, then $|\beta|\leq\lambda$. Moreover, if $|\beta|=\lambda$ for some eigenvalue $\beta$ of $B$, then $A=B$.
\vspace{.15cm}
\item[\textup{(5)}] There exist right (resp., left) eigenvectors $\textbf{v}$ and $\textbf{w}$ of $A$ with eigenvalue $\lambda$ such that all of the entries of $\textbf{v}$ and $\textbf{w}$ are strictly positive.
\end{enumerate}
\end{thm}
In particular, under certain conditions (which will hold for the clique graphs we study in this paper) the transfer matrix of a directed graph is primitive; see \cite[Theorem 8.5.3]{horn2012matrix}. In what follows, we say that $\mathcal{D}$ is \emph{strongly connected} if for every pair of vertices $v_i,v_j\in \mathcal{V}$ there exists a walk $\gamma$ on $\mathcal{D}$ such that $\iota(\gamma)=v_i$ and $\tau(\gamma)=v_j$.
\begin{lem}\label{lem:transfer+matrix+primitive} Let $\mathcal{D}$ be a directed multigraph with $n\geq2$ vertices. Moreover, assume that $\mathcal{D}$ is strongly connected and that there is an edge $v\rightarrow v$ for all vertices $v$ of $\mathcal{D}$. Then $A^{n-1}>0$, where $A$ is the transfer matrix of $\mathcal{D}$. Furthermore, the spectral radius $\textup{sr}(A)>1$.
\end{lem}
\begin{proof} Fix $1\leq i,j\leq n$ and let $\gamma$ be a walk of minimum length $\ell\geq0$ with $\iota(\gamma)=v_i$ and $\tau(\gamma)=v_j$. Note that $\ell\leq n-1$ since otherwise $\gamma$ revisits some vertex by the pigeonhole principle, violating the assumption that $\gamma$ has minimal length. Let $\epsilon_i=v_i\rightarrow v_i$. Then prepending $n-1-\ell\geq0$ copies of $\epsilon_i$ to $\gamma$ yields a walk of length $n-1$ from $v_i$ to $v_j$. But $(A^{n-1})_{ij}$ is the number of walks of length $n-1$ from $v_i$ to $v_j$ by \cite[Theorem 4.7.1]{StanleyEC1}. Thus, $(A^{n-1})_{ij}\geq1$ for all $i,j$ as claimed.

Finally, note that $B=I$ satisfies $0\leq B_{ij}\leq A_{ij}$ for all $i,j$ since every vertex of $\mathcal{D}$ has a loop. Moreover, $I\neq A$, since $A$ has nonzero off-diagonal entries (otherwise, $\mathcal{D}$ is not strongly connected). Hence, Theorem \ref{thm:P-F} part (4) implies that $1<\textup{sr}(A)$.
\end{proof}
Next, to analyze the poles of $Z_\eta$, we need the following facts from linear algebra.
\begin{lem}\label{lem:matrix+divergence} Let $M$ be an $n\times n$ nonnegative, primitive matrix and assume that $\textup{sr}(M)\geq1$. Moreover, let $\textbf{v},\textbf{w}\in\mathbb{R}^n$ be nonzero vectors such that $\textbf{v}$ has strictly positive entries and $\textbf{w}$ has nonnegative entries. Then the sequence
\[a_N(M,\textbf{v},\textbf{w}):=\textbf{v}^{\textup{T}}\Big(\sum_{m=0}^N M^m\Big)\textbf{w}\]
satisfies $\displaystyle{\lim_{N\rightarrow\infty}a_N(M,\textbf{v},\textbf{w})=\infty}$.
\end{lem}
\begin{proof} We first assume that $M>0$ has strictly positive entries. Let $\lambda=\textup{sr}(M)\geq1$ and let $\textbf{x}$ be an eigenvector of $M$ with eigenvalue $\lambda$ and strictly positive entries, guaranteed to exist by statement (5) of Theorem \ref{thm:P-F}. Now choose an entry $w_k$ of $\textbf{w}$ that is strictly positive. Then, for each $j$ we have that
\[(M\textbf{w})_j=\sum_{i=1}^nM_{ji}w_i\geq M_{jk}w_k>0,\]
since $M$ has strictly positive entries. Hence,
\[M\textbf{w}\geq \alpha\textbf{x}\;\;\;\text{for}\;\;\;\alpha=\min_{1\leq j\leq n} \frac{(M\textbf{w})_j}{x_j}>0.\]
Moreover, for $m\geq1$ we can multiply the inequality above by the nonnegative matrix $M^{m-1}$ to deduce that
\[M^m\textbf{w}=M^{m-1}(M\textbf{w})\geq M^{m-1}(\alpha\textbf{x})=\alpha\lambda^{m-1}\textbf{x}\geq \alpha\textbf{x}.\]
In particular, we see that
\[\textbf{v}^{\textup{T}}M^{m}\textbf{w}\geq \textbf{v}^{\textup{T}}(\alpha \textbf{x}) =\alpha (\textbf{v}^{\textup{T}}\textbf{x})=:c>0\]
holds for every $m\geq1$; here we use that both $\textbf{v}$ and $\textbf{x}$ have positive entries. Summing this inequality we see that
\begin{equation}\label{eq:sumlbd}
a_N=\textbf{v}^{\textup{T}}\Big(\sum_{m=0}^N M^m\Big)\textbf{w}=\textbf{v}^{\textup{T}}\textbf{w}+\textbf{v}^{\textup{T}}\Big(\sum_{m=1}^N M^m\Big)\textbf{w}\geq Nc.
\end{equation}
Hence $a_N\rightarrow\infty$ as $N\rightarrow\infty$ as claimed.

For the general case, suppose that $M$ is a primitive matrix with $\textup{sr}(M)\geq1$. Then we may choose $k>0$ such that $M^k>0$. Moreover, $\textup{sr}(M^k)=\textup{sr}(M)^k\geq1$. Hence, \eqref{eq:sumlbd} implies that
\[a_{kN}(M,\textbf{v},\textbf{w})=\textbf{v}^{\textup{T}}\Big(\sum_{m=0}^{kN} M^{m}\Big)\textbf{w}\geq \textbf{v}^{\textup{T}}\Big(\sum_{m=0}^N M^{km}\Big)\textbf{w}=a_N(M^k,\textbf{v},\textbf{w})\geq Nc.\]
Here $c$ is the constant established above for the matrix $M^k$ in place of $M$. In particular, since the sequence $a_{N}(M,\textbf{v},\textbf{w})$ is non-decreasing (it is a partial sum with nonnegative terms), it follows that $\displaystyle{\lim_{N\rightarrow\infty}a_N(M,\textbf{v},\textbf{w})=\infty}$ as claimed.
\end{proof}
Lastly, we need the following fact; see \cite[Theorem 8.4.5]{horn2012matrix} for a more general statement.
\begin{lem}\label{lem:critical+line+poles}
Let $A$ be an $n\times n$ nonnegative, primitive matrix with $\textup{sr}(A)=1$ and suppose that $B\in\textup{Mat}_{n\times n}(\mathbb{C})$ satisfies $|B_{ij}|=A_{ij}$ for all $1\leq i,j\leq n$. If $1$ is an eigenvalue of $B$, then $B=DAD^{-1}$ for some diagonal matrix $D$ with unimodular entries.
\end{lem}
\begin{proof} In what follows, given a complex valued matrix $W$ (or vector), we denote by $|W|$ the matrix (or vector) obtained by taking the complex absolute value of each entry; see \cite[\S8.1]{horn2012matrix} for basic properties. Now choose a nonzero eigenvector $\textbf{v}\in\mathbb{C}^n$ for $B$ satisfying $B\textbf{v}=\textbf{v}$. Then the triangle inequality implies that
\[|\textbf{v}|=|B\textbf{v}|\leq |B||\textbf{v}|=A\,|\textbf{v}| \]
entrywise. Now let $\textbf{u}$ be a left eigenvector of $A$ with eigenvalue $\lambda=1$ and all positive entries (given by Theorem \ref{thm:P-F} part (5)). Then
\[\textbf{u}^{\textup{T}}\big(A|\textbf{v}|-|\textbf{v}|\big)=(\textbf{u}^{\textup{T}}A)|\textbf{v}|-\textbf{u}^{\textup{T}}|\textbf{v}|=\textbf{u}^{\textup{T}}|\textbf{v}|-\textbf{u}^{\textup{T}}|\textbf{v}|=0.\]
Since $A|\textbf{v}|-|\textbf{v}|$ has nonnegative entries and $\textbf{u}$ has strictly positive entries, it follows that $A|\textbf{v}|=|\textbf{v}|$. Now choose $k>0$ such that $A^k$ has strictly positive entries (possible, since $A$ is primitive) and choose an index $i_0$ with $|\textbf{v}|_{i_0}\neq 0$ (possible, since eigenvectors are nonzero). Then multiplying the equation $A|\textbf{v}|=|\textbf{v}|$ by $A$ repeatedly yields $A^k|\textbf{v}|=|\textbf{v}|$. Hence,
\[|\textbf{v}|_j=\sum_{i=1}^n(A^k)_{ji}\,|\textbf{v}|_i\geq (A^k)_{ji_0}|\textbf{v}|_{i_0}>0\]
for all $1\leq j\leq n$. That is, $|\textbf{v}|$ is a vector with all positive entries. Now let $D$ be the diagonal matrix given by $D=\textup{Diag}(\textbf{v}_1/|\textbf{v}_1|,\dots,\textbf{v}_n/|\textbf{v}_n|)$. In particular, we note that $D$ has unimodular entries and $D|\textbf{v}|=\textbf{v}$ by definition of $D$. Finally, we set $C=D^{-1}BD$ and note that
\begin{equation}\label{eq:matrix+C}
C|\textbf{v}|=|\textbf{v}|\qquad\text{and}\qquad |C|=A.
\end{equation}
Indeed, $C|\textbf{v}|=D^{-1}BD|\textbf{v}|=D^{-1}B\textbf{v}=D^{-1}\textbf{v}=|\textbf{v}|$ and $C_{ij}=\zeta_j\zeta_i^{-1}B_{ij}$ where $\zeta_i=\textbf{v}_i/|{\textbf{v}_i}|$, from which it follows that $|C|=|B|=A$ by assumption on $B$. But then from \eqref{eq:matrix+C} we deduce that
\[0<|\textbf{v}|_j=\sum_{i=1}^nC_{ji}|\textbf{v}|_i=\Big|\sum_{i=1}^nC_{ji}|\textbf{v}|_i\Big|\leq \sum_{i=1}^n|C_{ji}|\,|\textbf{v}|_i=\sum_{i=1}^nA_{ji}|\textbf{v}|_i=(A|\textbf{v}|)_j=|\textbf{v}|_j\]
holds for all $j$. In particular, setting $w_{ji}:=C_{ji}|\textbf{v}|_i$ we deduce that
\[w_{j1}+\dots+w_{jn}=|w_{j1}|+\dots+|w_{jn}|\]
for all $j$. Since $\Re(w_{ji})\leq|w_{ji}|$ for each term, it follows that $w_{ji}=|w_{ji}|$ for all $1\leq i,j\leq n$. But then, since each entry of $|\textbf{v}|$ is nonzero, we deduce that $C_{ji}=|C_{ji}|=A_{ji}$ for all $1\leq i,j\leq n$. Thus $C=A$, so that $B=DAD^{-1}$, and the claim follows.
\end{proof}
As a consequence, we derive the analytic properties necessary to apply the Wiener-Ikehara Tauberian theorem.
\begin{thm} \label{thm:Z_eta+analytic+properties} Let $\mathcal{D}=(\mathcal{V},\mathcal{E},\mu)$ be a finite directed multigraph with $n\geq2$ vertices, let $d:\mathcal{V}\rightarrow \mathbb{Z}_{\geq2}$, and extend $d$ multiplicatively to walks on $\mathcal{D}$. Moreover, assume that $\mathcal{D}$ is strongly connected and that there is an edge $v\rightarrow v$ for all $v\in\mathcal{V}$. Then the following statements hold for all walks $\eta$ on $\mathcal{D}$:
\vspace{.1cm}
\begin{enumerate}
    \item[\textup{(1)}] There is a unique $\rho>0$
    with $\textup{sr}(M(\rho))=1$. In particular, $\det(I-M(\rho))=0$.
    \vspace{.2cm}
    \item[\textup{(2)}] $Z_\eta(s)$ converges for $\Re(s)>\rho$ and diverges for all real $s\leq\rho$. In particular, $\rho$ is the abscissa of convergence of $Z_\eta$.
    \vspace{.2cm}
    \item[\textup{(3)}] $Z_\eta$ extends meromorphically to $\mathbb{C}$, is analytic on $\{\Re(s)>\rho\}$, and has a simple pole with a positive residue at $s=\rho$.
    \vspace{.2cm}
    \item[\textup{(4)}] If in addition $d(v_1),\dots,d(v_n)$ are not all powers of a single integer, then $s=\rho$ is the unique pole of $Z_\eta(s)$ on the line $\Re(s)=\rho$.
\end{enumerate}
\end{thm}
\begin{proof}
Let $s\geq0$ be a real number. Then, since $A$ is primitive by Lemma \ref{lem:transfer+matrix+primitive}, $M(s)$ is primitive also. In fact, it is straightforward to check that if two nonnegative matrices $P$ and $Q$ have the same support (i.e., $P_{ij}>0$ if and only if $Q_{ij}>0$), then $P^m$ and $Q^m$ have the same support for all $m\geq1$. Now define the function $\lambda: \mathbb{R}_{\geq0}\rightarrow\mathbb{R}_{\geq0}$ given by $\lambda(s)=\textup{sr}(M(s))$. Then Theorem \ref{thm:P-F} part (4) implies that $\lambda$ is a strictly decreasing function: indeed, if $0\leq s<s'$, then $0\leq M_{ij}(s')\leq M_{ij}(s)$ for all $i,j$ and $M(s')\neq M(s)$, since each $d(v_j)>1$ and $A\neq0$. Moreover, $\lambda(0)=\textup{sr}(M(0))=\textup{sr}(A)>1$ by Lemma \ref{lem:transfer+matrix+primitive}. Similarly, $\lim_{s\rightarrow\infty} \lambda(s)=0$, since
\[\lambda(s)=\textup{sr}(M(s))\leq \|M(s)\|_\infty=\max_{1\leq i\leq n}\Big\{\sum_{j=1}^nA_{ij}d(v_j)^{-s}\Big\}\leq \|A\|_{\infty}\,d_{\min}^{-s},\]
where $d_{\min}$ and $\|A\|_{\infty}$ are as in Lemma \ref{lem:basic+convergence}. On the other hand, $\lambda$ is continuous since it is the composition of $s\mapsto M(s)$ and $M\mapsto\textup{sr}(M)$, which are both continuous. In particular, there exists a unique $s=\rho>0$ such that $\lambda(\rho)=1$. Moreover, $1$ is an eigenvalue of $M(\rho)$ by Theorem \ref{thm:P-F} part (1), so that $\det(I-M(\rho))=0$. Thus, we have established statement (1).

For statement (2), we first note that when $s$ is real, every term appearing in the derivation of \eqref{eq:Z+eta+expression} in the proof of Theorem \ref{thm:walk+with+terminal+segment} is nonnegative. Hence, the same computation (now justified since series of nonnegative terms may be rearranged freely) shows that
\begin{equation}\label{eq:Z}
Z_\eta(s)=d(\eta)^{-s}\Big(1+\delta(s)\Big(\sum_{m\geq0}
M(s)^{m}\Big)\nu(\eta)\Big)
\end{equation}
holds for all real $s$ as an identity in $[0,\infty]$. Now, for real $s>\rho$, we have that $\lambda(s)<1$ and so $\sum_{m\geq0}M(s)^m$ converges; see, for instance, \cite[Theorem 5.6.15]{horn2012matrix}. Hence, $Z_\eta(s)<\infty$ for such $s$. Moreover, a Dirichlet series with nonnegative coefficients that converges at a real point $s$ converges absolutely on $\Re(s')\geq s$. Thus, $Z_\eta$ converges on $\Re(s)>\rho$. On the other hand, $M(\rho)$ is a nonnegative, primitive matrix with $\textup{sr}(M(\rho))=1$, the vector $\delta(\rho)$ has strictly positive entries, and $\nu(\eta)$ has nonnegative entries and is nonzero, since its $q$th entry is $A_{qq}\geq1$ (here $v_q=\iota(\eta)$). Therefore, Lemma \ref{lem:matrix+divergence} applied to the matrix $M(\rho)$ and the vectors $\textbf{v}=\delta(\rho)^{\textup{T}}$ and $\textbf{w}=\nu(\eta)$, together with \eqref{eq:Z}, implies that $Z_\eta(\rho)=\infty$. Finally, let $s<\rho$ be real. Since $d(\gamma)\geq2$ for every walk $\gamma$ and $\rho-s>0$, we have that
\[d(\gamma)^{-s}=d(\gamma)^{-\rho}\,d(\gamma)^{\rho-s}\geq d(\gamma)^{-\rho}\]
for every walk $\gamma$. Hence, for every $x\geq1$,
\[\sum_{k\leq x}\frac{a_k}{k^{s}}=\sum_{\substack{\gamma\in\mathcal{W}_\eta\\ d(\gamma)\leq x}}d(\gamma)^{-s}\;\geq\sum_{\substack{\gamma\in\mathcal{W}_\eta\\ d(\gamma)\leq x}}d(\gamma)^{-\rho}=\sum_{k\leq x}\frac{a_k}{k^{\rho}}.\]
Since $Z_\eta(\rho)=\infty$ and all terms are nonnegative, the right-hand side tends to $\infty$ as $x\rightarrow\infty$. Therefore, so does the left-hand side, i.e., $Z_\eta(s)=\infty$ as well. Hence, $Z_\eta$ diverges for all real $s\leq\rho$, and $\rho$ is the abscissa of convergence of $Z_\eta$ for all walks $\eta$.

Next, for statement (3), let $s\in\mathbb{C}$ with $\sigma=\Re(s)>\rho$, and for a complex matrix $W$ write $|W|$ for the matrix of absolute values of the entries of $W$. Since $|d(v_j)^{-s}|=d(v_j)^{-\sigma}$, we have that $|M(s)|=M(\sigma)$, and hence $|M(s)^m|\leq M(\sigma)^m$ entrywise for all $m\geq0$. Since $\sum_{m\geq0}M(\sigma)^m$ converges by the proof of statement (2), it follows that $\sum_{m\geq0}M(s)^m$ converges absolutely, and the standard telescoping argument identifies it with $(I-M(s))^{-1}$. In particular, $\det(I-M(s))\neq0$ for all $\Re(s)>\rho$. Moreover, every rearrangement in \eqref{eq:Z+eta+expression} is dominated by the corresponding (convergent) sum at the real point $\sigma$, so the computation in the proof of Theorem \ref{thm:walk+with+terminal+segment} and Cramer's rule give
\begin{equation}\label{eq:Cramer}
Z_{\eta}(s)=d(\eta)^{-s}\Big(1+\frac{\delta(s)\,\textup{adj}\big(I-M(s)\big)\,\nu(\eta)}{\det\big(I-M(s)\big)}\Big)
\end{equation}
for all $\Re(s)>\rho$. In particular, $Z_\eta(s)$ is a rational function in $d(v_1)^{-s},\dots,d(v_n)^{-s}$, and so has a meromorphic continuation to $\mathbb{C}$; this continuation is analytic on $\Re(s)>\rho$, since the denominator in \eqref{eq:Cramer} does not vanish there. Moreover, since $\rho$ is the abscissa of convergence of $Z_\eta$ by statement (2), Landau's theorem (see \cite[Theorem 1.7]{montgomery2007multiplicative}) implies that $s=\rho$ is a singularity, and hence a pole, of the continuation. Therefore, it suffices to prove that $\rho$ is a simple pole with positive residue. In particular, we first show that $s=\rho$ is a simple root of the analytic function $p(s)=\det(I-M(s))$. To see this, we note that
\begin{equation}\label{eq:Jacobi}
p'(\rho)=\textup{Tr}\bigg(\textup{adj}\Big(I-M(\rho)\Big)\,\bigg(\hspace{-.1cm}-\frac{dM(s)}{ds}\bigg)\bigg\vert_{s=\rho}\ \bigg)
\end{equation}
by Jacobi's formula for the derivatives of determinants. Now let $\textbf{v}$ and $\textbf{w}$ be the positive right and left eigenvectors of the primitive matrix $M(\rho)$ with eigenvalue $\lambda(\rho)=1$; see Theorem \ref{thm:P-F} part (5). Then \cite[Theorem 2.1]{castillo2022formula} implies that
\begin{equation}\label{eq:adj}
\textup{adj}\Big(I-M(\rho)\Big)=c(\rho)\,\textbf{v}\textbf{w}^{\textup{T}},\qquad\text{where}\qquad c(\rho):=\frac{\prod_{j=2}^n\big(1-\lambda_j(\rho)\big)}{\textbf{w}^{\textup{T}}\textbf{v}}
\end{equation}
and $1=\lambda_1(\rho), \lambda_2(\rho),\dots, \lambda_n(\rho)$ are the eigenvalues of $M(\rho)$, listed with multiplicity. Moreover, $c(\rho)>0$. Indeed, by Theorem \ref{thm:P-F} parts (2) and (3), we have $|\lambda_j(\rho)|<1$ for all $j\geq2$; since $M(\rho)$ is a real matrix, its non-real eigenvalues come in conjugate pairs, so that each real factor $1-\lambda_j(\rho)$ and each product of a conjugate pair of factors $|1-\lambda_j(\rho)|^2$ is positive; and $\textbf{w}^{\textup{T}}\textbf{v}>0$ since $\textbf{v}$ and $\textbf{w}$ have positive entries. On the other hand, since $M_{ij}(s)=A_{ij}d(v_j)^{-s}$ by definition, we see that
\begin{equation}\label{eq:M+derivative}
\bigg(\hspace{-.1cm}-\frac{dM(s)}{ds}\bigg\vert_{s=\rho}\,\bigg)_{ij}= M_{ij}(\rho)\log(d(v_j)).
\end{equation}
In particular, combining \eqref{eq:Jacobi}, \eqref{eq:adj}, and \eqref{eq:M+derivative} we see that
\[p'(\rho)=\textup{Tr}\Big(c(\rho)\,(\textbf{v}\textbf{w}^{\textup{T}})M(\rho)\,\textup{Diag}(\ell_1,\dots,\ell_n)\Big),\]
where $\ell_i:=\log(d(v_i))>0$. However, $\textbf{w}^{\textup{T}}M(\rho)=\textbf{w}^{\textup{T}}$ by definition of $\textbf{w}$. In particular, we compute from above that
\[p'(\rho)=c(\rho)\,\textbf{w}^{\textup{T}}\textup{Diag}(\ell_1,\dots,\ell_n)\textbf{v}.\]
Here we use standard properties of the trace: that $\textup{Tr}(cX)=c\,\textup{Tr}(X)$, that $\textup{Tr}(\textbf{b}\textbf{a}^{\textup{T}})=\textbf{a}^{\textup{T}}\textbf{b}$, and that $\textup{Tr}(XYW)=\textup{Tr}(WXY)=\textup{Tr}(YWX)$. In particular, $p'(\rho)>0$ since $c(\rho)>0$ and all matrices and vectors above have strictly positive entries. Therefore, $s=\rho$ is a simple zero of $p$.

Finally, by \eqref{eq:Cramer} we may write $Z_\eta(s)=f(s)/p(s)$, where
\[f(s)=d(\eta)^{-s}\Big(p(s)+\delta(s)\,\textup{adj}\big(I-M(s)\big)\,\nu(\eta)\Big)\]
is an entire function. For real $s>\rho$, all eigenvalues of the real matrix $M(s)$ have absolute value at most $\lambda(s)<1$, so the argument used to show $c(\rho)>0$ gives $p(s)=\prod_{j=1}^n(1-\lambda_j(s))>0$; since also $Z_\eta(s)>0$, it follows that $f(s)=Z_\eta(s)p(s)>0$ for all real $s>\rho$. Hence, $f(\rho)\geq0$ by continuity. On the other hand, since $s=\rho$ is a pole of $Z_\eta$ and $p'(\rho)\neq0$, we have that
\[\operatorname{ord}_{\rho}(f)-1=\operatorname{ord}_{\rho}(f)-\operatorname{ord}_{\rho}(p)=\operatorname{ord}_{\rho}(Z_\eta)\leq-1.\]
Hence, it must be the case that $\operatorname{ord}_{\rho}(f)=0$, so that $f(\rho)\neq0$ and thus $f(\rho)>0$. Therefore, $s=\rho$ is a simple pole of $Z_\eta$ and
\[\textup{Res}(Z_\eta,\rho)=\frac{f(\rho)}{p'(\rho)}>0,\]
as claimed.

Finally, for statement (4) we assume further that $d(v_1), \dots, d(v_n)$ are not all powers of a single integer and that $s=\rho+iy\in\mathbb{C}$ is a pole of $Z_\eta$. Then Theorem \ref{thm:walk+with+terminal+segment} implies that $\det(I-M(s))=0$. Hence, $1$ is an eigenvalue of $M(s)$. On the other hand, since $|d(v_j)^{-s}|=d(v_j)^{-\rho}$, we see that $|M(s)|=M(\rho)$. In particular, the matrices $M(s)$ and $M(\rho)$ satisfy the hypotheses of Lemma \ref{lem:critical+line+poles} (with $B=M(s)$ and $A=M(\rho)$ in the notation of the lemma), so that $M(s)=DM(\rho)D^{-1}$ for some diagonal matrix $D$ with unimodular entries. But then, equating diagonal entries in this expression implies that
\[A_{ii}d(v_i)^{-s}=A_{ii}d(v_i)^{-\rho}\]
holds for all $1\leq i\leq n$. Since $A_{ii}>0$ by our assumption about the existence of loops at each vertex, we see that $d(v_i)^{-s}=d(v_i)^{-\rho}$, i.e., $d(v_i)^{iy}=1$, for all $i$. In particular, if $s\neq \rho$, then $y\neq0$ and Lemma \ref{lem:mult_dep} below implies that $d(v_1), \dots, d(v_n)$ are all powers of a single integer, a contradiction. Therefore, $s=\rho$ is the unique pole of $Z_\eta$ on the line $\Re(s)=\rho$ as claimed.
\end{proof}
\begin{lem}\label{lem:mult_dep} Let $d_1,\dots,d_r\geq2$ be integers. If there exists a nonzero $y\in\mathbb{R}$ such that $d_j^{iy}=1$ for all $1\leq j\leq r$, then $\log(d_i)/\log(d_j)\in\mathbb{Q}$ for all $1\leq i,j\leq r$. In particular, $d_1,\dots, d_r$ are pairwise multiplicatively dependent and $\langle d_1,\dots, d_r\rangle\subseteq \langle d\rangle$ for some integer $d\geq2$.
\end{lem}
\begin{proof} We have that $d_j^{iy}=e^{iy\log d_j}$. Hence, $d_j^{iy}=1$ if and only if $y\log d_j=2\pi k_j$ for some $k_j\in\mathbb{Z}$. Moreover, $k_j\neq0$ for all $j$, since $y\neq0$ and $\log d_j>0$. Hence, $\log(d_i)/\log(d_j)=k_i/k_j$ is rational for all pairs $i,j$. In particular, it follows that $d_i^{k_j}=d_j^{k_i}$, and each pair $d_i, d_j$ is multiplicatively dependent. From here, it follows from \cite[Lemma 2.4]{MR4988033} that $\langle d_1,\dots, d_r\rangle\subseteq \langle d\rangle$ for some $d\geq2$.
\end{proof}
Finally, the desired asymptotic for the number of walks with a given terminal segment follows from the Wiener-Ikehara Tauberian theorem \cite[\S8.3]{montgomery2007multiplicative}.
\begin{cor}\label{cor:count+walks+terminal+segment} Let $\mathcal{D}=(\mathcal{V},\mathcal{E},\mu)$ be a finite directed multigraph with $n\geq2$ vertices, let $d:\mathcal{V}\rightarrow \mathbb{Z}_{\geq2}$, and extend $d$ multiplicatively to walks on $\mathcal{D}$. Moreover, assume that $\mathcal{D}$ is strongly connected, that there is an edge $v\rightarrow v$ for all $v\in\mathcal{V}$, and that $d(v_1),\dots,d(v_n)$ are not all powers of a single integer. Then there exists a constant $\rho>0$ such that for all walks $\eta$ of $\mathcal{D}$ there is a constant $c_\eta>0$ such that
\[\#\{\gamma\in\mathcal{W}_\eta: d(\gamma)\leq x\}=c_\eta x^{\rho}+o(x^{\rho})\]
as $x\rightarrow\infty$.
\end{cor}
\begin{proof} Let $\rho$ be as in Theorem \ref{thm:Z_eta+analytic+properties}, fix a walk $\eta$, and let $R_\eta=\textup{Res}(Z_\eta,\rho)>0$. By statements (2)--(4) of Theorem \ref{thm:Z_eta+analytic+properties}, $Z_\eta(s)=\sum_{k\geq1}a_kk^{-s}$ is a Dirichlet series with nonnegative coefficients that converges for $\Re(s)>\rho$, is analytic at every point of the line $\Re(s)=\rho$ other than $s=\rho$, and satisfies $Z_\eta(s)=R_\eta/(s-\rho)+H(s)$ near $s=\rho$ for some function $H$ analytic at $\rho$. Hence, Delange's generalization of the Wiener--Ikehara theorem \cite{delange1954generalisation} (see also \cite[Ch.~III, \S3]{narkiewicz1983number}) implies that
\[\#\{\gamma\in\mathcal{W}_\eta: d(\gamma)\leq x\}=\sum_{k\leq x}a_k=\frac{R_\eta}{\rho}\,x^{\rho}+o(x^{\rho})\]
as $x\rightarrow\infty$. Thus the claim holds with $c_\eta=R_\eta/\rho>0$. Alternatively, one may deduce this from the case $\rho=1$ in \cite[Corollary 8.8]{montgomery2007multiplicative} by applying it to $\sum_{k\geq1}a_kk^{1-\rho}k^{-s}=Z_\eta(s+\rho-1)$ and then using partial summation.
\end{proof}
Since every walk has a unique terminal vertex, we can easily deduce an asymptotic for the total number of walks of bounded weight by partitioning walks by their terminal vertex and then applying Corollary \ref{cor:count+walks+terminal+segment} to the walks of length $0$.
\begin{cor}\label{cor:count+total+walks} Let $\mathcal{D}=(\mathcal{V},\mathcal{E},\mu)$ be a finite directed multigraph with $n\geq2$ vertices, let $d:\mathcal{V}\rightarrow \mathbb{Z}_{\geq2}$, and extend $d$ multiplicatively to walks on $\mathcal{D}$. Moreover, assume that $\mathcal{D}$ is strongly connected, that there is an edge $v\rightarrow v$ for all $v\in\mathcal{V}$, and that $d(v_1),\dots,d(v_n)$ are not all powers of a single integer. Then there exist constants $c,\rho>0$ such that
\[\#\{\textup{walks $\gamma$}\,:\, d(\gamma)\leq x\}=cx^{\rho}+o(x^{\rho})\]
as $x\rightarrow\infty$.
\end{cor}
\begin{proof} Let $\eta_i$ be the walk of length $0$ consisting of the vertex $v_i\in\mathcal{V}$, and let $\mathcal{W}$ be the set of all finite walks in $\mathcal{D}$. Then, since every walk has terminal segment $\eta_i$ for a unique $1\leq i\leq n$ (namely, the index of its terminal vertex), we have that $\mathcal{W}=\mathcal{W}_{\eta_1}\sqcup\dots\sqcup \mathcal{W}_{\eta_n}$. Now let $c_i=c_{\eta_i}$ be the constants from Corollary \ref{cor:count+walks+terminal+segment}. Then $c:=\sum_{i=1}^n c_i>0$ has the desired property.
\end{proof}

\section{Partially Commutative Semigroups}\label{sec:trace}
\subsection{The Cartier-Foata Normal Form}
In this section, we give a brief introduction to the theory of \emph{partially commutative semigroups}, focusing primarily on the material needed to introduce the Cartier-Foata normal form \cite{MR239978} and the corresponding directed graph of cliques; see \cite[\S8.2]{rozenberg2012handbook} or \cite{diekert1995book} for more complete introductions. In particular, we note that the associated monoids appear in the literature as free partially commutative monoids, trace monoids, Cartier–Foata monoids, graph monoids, or right-angled Artin monoids.   

Let $\mathcal{G}$ be a finite simple graph (no loops or multiple edges) with vertex set $V(\mathcal{G})=\{f_1,\dots,f_r\}$ and edge set $E(\mathcal{G})$. We define an associated semigroup $\mathcal{S}(\mathcal{G})$ given by the presentation,    
\[\mathcal{S}=\mathcal{S}(\mathcal{G}):=\Big\langle f_1,\dots,f_r\,\Big\vert\,\, f_if_j=f_jf_i\;\, \text{whenever}\;\, \{f_i,f_j\}\in E(\mathcal{G})  \Big\rangle,\]
and call $\mathcal{S}$ the \emph{partially commutative semigroup} associated to $\mathcal{G}$ and call $\mathcal{G}$ the \emph{commutation graph} of $S$. We often simply refer to the pair $(\mathcal{S},\mathcal{G})$ as a partially commutative semigroup. Formally, the construction of $\mathcal{S}$ is as follows: let $A=\{f_1,\dots,f_r\}$ be an alphabet, let $A^+$ be the \emph{free semigroup} generated by $A$, and let $\equiv_{R}$ denote the congruence (equivalence relation compatible with concatenation) on $A^+$ generated by the relation 
\[R=R(\mathcal{G}):=\Big\{\big(f_i f_j,f_j f_i\big)\in A^+\times A^+: \{f_i,f_j\}\in E(\mathcal{G})\Big\}.\] 
Then $\mathcal{S}$ is the quotient semigroup $\mathcal{S}:=A^+/\equiv_R$; see, for instance, \cite[\S1.5-\S1.6]{howie_fundamentals} for introductions to congruences (generated by relations) and presentations of semigroups. 

\begin{remark}\label{rmk:G+vs+barG}
In the literature, the complement graph $\overline{\mathcal{G}}$, which determines an associated symmetric and reflexive (dependence) relation, is instead used to define $\mathcal{S}$. Nevertheless, since the semigroups we study in this paper consist of self maps of a variety, and given the ubiquity of commuting maps in dynamics (going back to Fatou \cite{fatou1923substitutions}, Ritt \cite{ritt1923permutable}, and Julia \cite{julia1922memoire}), we find the commutation graph $\mathcal{G}$ a more natural object for our purposes. 
\end{remark}
Likewise, for both reasons of convenience and to match preexisting literature, where monoids (semigroups with an identity element) are used instead of semigroups, we let 
\[\mathcal{M}=\mathcal{M}(\mathcal{G}):=\mathcal{S}(\mathcal{G})^1=\mathcal{S}(\mathcal{G})\sqcup\{1\}\] 
and declare that $1\cdot F=F=F\cdot1$ for all $F\in \mathcal{M}$. In particular, $\mathcal{M}(\mathcal{G})$ is a monoid containing $\mathcal{S}=\mathcal{S}(\mathcal{G})$, and we call the pair $(\mathcal{M},\mathcal{G})$ the \emph{partially commutative monoid} with commutation graph $\mathcal{G}$. As with $\mathcal{S}$, we may identify $\mathcal{M}$ as a quotient, $\mathcal{M}\cong A^*/\equiv_R$, where $A^*$ is the free monoid generated by $A$ and $\equiv_R$ is the congruence generated by the relation $R$, in this case viewed as a relation on $A^*$. This identification holds since the relations in $R$ preserve length; in particular, the empty word $\epsilon$ is alone in its class and $\mathcal{S}$ has no identity element, so $\mathcal{M}$ coincides with what is written $\mathcal{S}^1:=\mathcal{S}\sqcup\{1\}$ in the literature and in Section \ref{sec:semigroups}.    
\begin{remark}\label{rmk:graph+ops}
We note the following simple facts relating graph operations and semigroup operations. Let $\mathcal{G}_1$ and $\mathcal{G}_2$ be finite simple graphs on disjoint vertex sets, write $\mathcal{G}_1\sqcup\mathcal{G}_2$ for their disjoint union and $\mathcal{G}_1\vee\mathcal{G}_2$ for their join (the disjoint union together with every edge joining $V(\mathcal{G}_1)$ to 
$V(\mathcal{G}_2)$). Then 
\[
\mathcal{M}(\mathcal{G}_1\sqcup\mathcal{G}_2)\;\cong\;\mathcal{M}(\mathcal{G}_1)*\mathcal{M}(\mathcal{G}_2)
\qquad\text{and}\qquad 
\mathcal{M}(\mathcal{G}_1\vee\mathcal{G}_2)\;\cong\;\mathcal{M}(\mathcal{G}_1)\times \mathcal{M}(\mathcal{G}_2),
\]
where $*$ and $\times$ denote the free product and the direct product of monoids. Both are immediate from the presentations: in the first case $R(\mathcal{G}_1\sqcup\mathcal{G}_2)=R(\mathcal{G}_1)\sqcup R(\mathcal{G}_2)$, so that a homomorphism from $\mathcal{M}(\mathcal{G}_1\sqcup\mathcal{G}_2)$ to a monoid $N$ is precisely a pair of homomorphisms $\mathcal{M}(\mathcal{G}_i)\to N$, which is the universal property of the free product; in the second case the added relations sort each word uniquely into a product of a word in $A_1^*$ and a 
word in $A_2^*$. Both identities are instances of the theory of \emph{graph products}, of which the free product and the direct product are the extreme cases (corresponding to edgeless and complete graphs); see \cite{dandan2023graph}.
\end{remark}
Next, we call a non-empty subset $\mathfrak{c}\subseteq V(\mathcal{G})$ a \emph{clique} if $\{f,f'\}\in E(\mathcal{G})$ for all $f\neq f'$ in $\mathfrak{c}$. Equivalently, cliques correspond to the set of vertices of a complete subgraph of $\mathcal{G}$. Moreover, since the elements of a clique $\mathfrak{c}=\{f_{i_1},\dots,f_{i_k}\}$ commute, we let $\bar{\mathfrak{c}}=f_{i_1}\cdots f_{i_k}\in \mathcal{S}$ be the corresponding element of $\mathcal{S}$, which is formally the image of the word $f_{i_1}\cdots f_{i_k}\in A^+$ in the quotient $\mathcal{S}=A^+/\equiv_R$. Lastly, we say that an ordered pair of cliques $(\mathfrak{c},\mathfrak{c}')$ is \emph{admissible} if for all $f'\in\mathfrak{c}'$ there is an $f\in \mathfrak{c}$ such that $\{f,f'\}\not\in E(\mathcal{G})$. In this case, we write $\mathfrak{c}\rightarrow\mathfrak{c}'$ when the pair $(\mathfrak{c},\mathfrak{c}')$ is admissible. With these notions in place, we have the following fundamental result of Cartier-Foata \cite[Theorem 1.2]{MR239978}, which says that every element of $\mathcal{S}$ can be written uniquely as a product of admissible cliques: 
\begin{thm}\label{thm:CF+NF}[Cartier-Foata Normal-Form] Let $(\mathcal{S},\mathcal{G})$ be a partially commutative semigroups and let $\Gamma$ be the finite set of cliques of $\mathcal{G}$. Then every element $F\in \mathcal{S}$ can be written uniquely in the form
\begin{equation}\label{eq:CF+NF}
F=\overline{\mathfrak{c}_1}\cdots \overline{\mathfrak{c}_n}
\end{equation}
for some $n\geq1$ and some $\mathfrak{c}_1,\dots,\mathfrak{c}_n\in\Gamma$ with $\mathfrak{c}_i\rightarrow\mathfrak{c}_{i+1}$ for all $1\leq i\leq n-1$ when $n\geq2$.    
\end{thm}
In particular, the Cartier-Foata normal form supplies the bridge we use to apply the main results of Section \ref{sec:d+to+h} and Section \ref{sec:counting+walks}.  
\begin{prop}\label{prop:CF+implies+suffix+closed+directed+graph+walks} Let $(\mathcal{S},\mathcal{G})$ be a partially commutative semigroups and let $\Gamma$ be the finite set of cliques of $\mathcal{G}$. Then the following statements hold: 
\vspace{.05cm} 
\begin{enumerate}
\item[\textup{(1)}] The set $\mathcal{N}\subseteq\Gamma^*$ of all possible words $\mathfrak{c}_1\dots\mathfrak{c}_n$ with $\mathfrak{c}_i\rightarrow\mathfrak{c}_{i+1}$ for all $1\leq i\leq n-1$ when $n\geq2$ together with the empty word is a suffix closed normal form on $\mathcal{M}(\mathcal{S})$.\vspace{.2cm}  
\item[\textup{(2)}] Let $\mathcal{D}=(\Gamma,\rightarrow)$ be the directed graph of cliques with edges $\mathfrak{c}\rightarrow \mathfrak{c}'$ given by admissible pairs. Then the map $F\rightarrow w(F)$ given by 
\[w(F)=(\mathfrak{c}_1,\dots,\mathfrak{c}_n)\]
and induced by \eqref{eq:CF+NF} gives a bijection between $\mathcal{S}$ and non-empty walks in $\mathcal{D}$. 
\vspace{.15cm} 
\item[\textup{(3)}] Every vertex $\mathfrak{c}$ of $\mathcal{D}$ has a loop, i.e., $\mathfrak{c}\rightarrow\mathfrak{c}$ is admissible for all $\mathfrak{c}\in\Gamma$.
\vspace{.15cm} 
\item[\textup{(4)}] If the complement of the commutation graph $\overline{\mathcal{G}}$ is connected, then the directed graph $\mathcal{D}$ of cliques is strongly connected.   
\end{enumerate}
\end{prop}
\begin{remark}\label{rmk:edge+drop} Since $\mathcal{D}$ has at most one edge from each pair of vertices, a walk in the sense of Section \ref{sec:counting+walks} is determined completely by its vertex sequence alone, and we suppress the edge sequence when defining $w(F)$ in Proposition \ref{prop:CF+implies+suffix+closed+directed+graph+walks}.    
\end{remark}
\begin{proof}
Statement (1) follows from the bijective correspondence in Theorem \ref{thm:CF+NF} and the fact that if $\mathfrak{c}_1\rightarrow \dots\rightarrow\mathfrak{c}_n$ is an admissible sequence for some $n\geq2$ then $\mathfrak{c}_k\rightarrow \dots\rightarrow\mathfrak{c}_n$ is an admissible sequence for all $1\leq k\leq n-1$. Statement (2) is simply a restatement of Theorem \ref{thm:CF+NF}; see also Remark \ref{rmk:edge+drop} above. For statement (3), note that if $f\in\mathfrak{c}$ then $\{f,f\}\not\in E(\mathcal{G})$, since $\mathcal{G}$ is assumed to be a simple graph (so $\mathcal{G}$ has no loops). Hence, $\mathfrak{c}\rightarrow\mathfrak{c}$ is admissible by definition. Lastly, statement (4) is well known; see for instance \cite[Lemma 3.2]{KrobMairesseMichos2003}.       
\end{proof}
\subsection{Proof of Main Theorem}
We now have the pieces in place to prove our main result from the Introduction \ref{sec:intro}. In what follows, we assume that $(\mathcal{S},\mathcal{G})$ is a partially commutative semigroup, that $(\mathcal{D},\rightarrow)$ is the corresponding directed graph of cliques, and that $d: \mathcal{S}\rightarrow\mathbb{Z}_{\geq2}$ is a multiplicative function. Moreover, we assume that the complement graph $\overline{\mathcal{G}}$ is connected and that the images $d(\mathcal{S})$ is not contained in a cyclic, multiplicative subsemigroup of $\mathbb{Z}_{\geq2}$. Concretely, if $\mathcal{S}$ is generated by $f_1,\dots, f_r$, then we assume that the $d(f_1),\dots,d(f_r)$ are not all powers of a single integer $d\geq2$. 
\begin{remark}\label{rmk:at+least+2+cliques} In particular, it must be the case that $\mathcal{D}$ has at least $2$ vertices. For if not, then $\mathcal{D}$ has a single vertex. Equivalently, a generating set of $\mathcal{S}$ forms a single clique. But then $\mathcal{S}$ is commutative and $\mathcal{G}(\mathcal{S})=K_r$ for some $r\geq1$. However, in this case $\overline{\mathcal{G}(\mathcal{S})}=\overline{K_r}=E_r$ where $E_r$ is the empty graph with $r$ vertices and no edges. This is clearly disconnected unless $r=1$. But then $d(\mathcal{S})=\langle d(f)\rangle$ for some generator $f$ of $\mathcal{S}$, a contradiction.        
\end{remark}
\begin{proof}[(Proof of Theorem \ref{thm:main})] 
Note first that we may extend $d$ to a function on cliques (i.e., the vertices of $\mathcal{D}$), by defining $d(\mathfrak{c}):=d(\bar{\mathfrak{c}})=\prod_{f\in\mathfrak{c}}d(f)$. Likewise, we may extend $d$ to walks in $\mathcal{D}$ multiplicatively as in Section \ref{sec:counting+walks}. Now for $F\in \mathcal{S}$ we let $w(F)$ be the walk in $\mathcal{D}$ given by the Cartier-Foata normal form of $F$ as in Proposition \ref{prop:CF+implies+suffix+closed+directed+graph+walks} part (2). Then we see that 
\begin{equation}\label{eq:trivial+d}
d(F)=d(\overline{\mathfrak{c}_1}\cdots \overline{\mathfrak{c}_n})=\prod_{i=1}^nd(\overline{\mathfrak{c}_i})=\prod_{i=1}^nd(\mathfrak{c}_i)=d(w(F))
\end{equation}
since $d$ is multiplicative (and by definition of $d$ on the various objects). In particular, under the bijection in Proposition \ref{prop:CF+implies+suffix+closed+directed+graph+walks} part (2), we have that 
\begin{equation}\label{eq:count+in+S+eq+count+walks}
\#\{F\in \mathcal{S}\,:\, d(F)\leq x\}=\#\{\text{walks $w$ in $\mathcal{D}$}\,:\, d(w)\leq x\}. 
\end{equation}
On the other hand, parts (3) and (4) of Proposition \ref{prop:CF+implies+suffix+closed+directed+graph+walks} and Corollary \ref{cor:count+total+walks} imply that
\begin{equation}\label{eq:asymp+walks}
    \#\{\text{walks $w$ in $\mathcal{D}$}\,:\, d(w)\leq x\}=cx^{\rho}+o(x^{\rho})
\end{equation}
for some constants $c,\rho>0$. In particular, parts (1), (2), and (3) of Theorem \ref{thm:main} follow from \eqref{eq:count+in+S+eq+count+walks}, \eqref{eq:asymp+walks}, and the fact that the corresponding maps on endomorphisms are multiplicative.

As for the proof of the claim about heights, we note that $\mathcal{S}$ has a suffix closed normal form by Proposition \ref{prop:CF+implies+suffix+closed+directed+graph+walks} part (1). Moreover, if $\mathcal{S}$ is a polarized set of endomorphisms on a variety $V$ with respect to the ample sheaf $\mathcal{L}$ and $P\in V$ is not a preperiodic point, then the pair of functions $(d,h)$ given by $F^*(\mathcal{L})\cong \mathcal{L}^{\otimes d(F)}$ and $h(F)=h_{\mathcal{L}}(F(P))$ satisfy hypotheses (1) and (2) of Theorem \ref{thm:d+to+h} by Lemma \ref{lem:(d,h)+preperiodic}. Therefore, it remains to show that hypothesis (3) of Theorem \ref{thm:d+to+h} holds to prove the theorem. 

To see this, we show that $\mathcal{C}_G$ from Section \ref{sec:d+to+h} is in bijection with $\mathcal{W}_{w(G)}$ from Section \ref{sec:counting+walks}. This is essentially a tautology, up to a transfer of language. Suppose first that $G\in \mathcal{S}$ and write $G=\overline{\mathfrak{c}_1}\cdots \overline{\mathfrak{c}_k}$ for some unique cliques $\mathfrak{c}_i$ satisfying $\mathfrak{c}_i\rightarrow \mathfrak{c}_{i+1}$ using the Cartier-Foata normal form. For consistency with Section \ref{sec:d+to+h}, we write $\textup{nf}(G)=\mathfrak{c}_1\cdots \mathfrak{c}_k\in \Gamma^+$. Then, if $F\in \mathcal{S}$ satisfies $\textup{nf}(G)\leq \textup{nf}(F)$ with $F\neq G$, we have that 
\[F=\overline{\mathfrak{e}_1}\cdots \overline{\mathfrak{e}_m}\;\overline{\mathfrak{c}_1}\cdots \overline{\mathfrak{c}_k}\]
for some unique cliques $\mathfrak{e}_j$ satisfying $\mathfrak{e}_j\rightarrow \mathfrak{e}_{j+1}$ and $\mathfrak{e}_m\rightarrow\mathfrak{c}_1$. Hence, 
\[w(F)=(\mathfrak{e}_1,\dots,\mathfrak{e}_m,\mathfrak{c}_1,\dots,\mathfrak{c}_k)=w(\mathfrak{e}_1,\dots,\mathfrak{e}_m)\cdot w(G)\in \mathcal{W}_{w(G)},\]
where $\cdot$ is the concatenation operation on walks and $\mathcal{W}_{w(G)}$ is the set of walks with terminal segment $w(G)$, both defined in Section \ref{sec:counting+walks}. Similarly, if $\gamma\in \mathcal{W}_{w(G)}\setminus\{w(G)\}$ then we can write $\gamma=\omega\cdot w(G)$ for some non-empty walk $\omega=(\mathfrak{e}_1,\dots,\mathfrak{e}_m)$ with edges $\mathfrak{e}_j\rightarrow \mathfrak{e}_{j+1}$ and  $\mathfrak{e}_m\rightarrow\mathfrak{c}_1$. Then it follows that 
\[w^{-1}(\gamma)=\overline{\mathfrak{e}_1}\cdots \overline{\mathfrak{e}_m}\;\overline{\mathfrak{c}_1}\cdots \overline{\mathfrak{c}_k}\in\mathcal{C}_G,\]
since $\textup{nf}(w^{-1}(\gamma))=\mathfrak{e}_1\dots\mathfrak{e}_m\mathfrak{c}_1\dots\mathfrak{c}_k\in\Gamma^+$ (by the uniqueness of normal form) so $\textup{nf}(G)=\mathfrak{c}_1\dots\mathfrak{c}_k\leq \textup{nf}(F)$. Hence, we deduce that the restriction map, 
\[w\big\vert_{\mathcal{\mathcal{C}_G}}:\mathcal{C}_G\rightarrow\mathcal{W}_{w(G)},\]
is a bijection. In particular, \eqref{eq:trivial+d} and Corollary \ref{cor:count+walks+terminal+segment} together imply that 
\[\#\{F\in \mathcal{C}_G\,:\, d(F)\leq x\}=\{\gamma\in \mathcal{W}_{w(G)}\,:\, d(\gamma)\leq x\}=c_Gx^{\rho}+o(x^{\rho})\]
for some $c_G>0$ and all $G\in \mathcal{S}$; here we use that $\mathcal{D}$ is strongly connected and has loops by (3) and (4) of Proposition \ref{prop:CF+implies+suffix+closed+directed+graph+walks}. Therefore, all three hypotheses of Theorem \ref{thm:d+to+h} are met, so that 
\[\{F\in \mathcal{S}\,:\, h_{\mathcal{L}}(F(P))\leq x\}=c(P,\mathcal{L})x^{\rho}+o(x^{\rho})\]
for some constant $c(P,\mathcal{L})>0$ as claimed; here the constant $c(P,\mathcal{L})>0$ depends on $P$ and $\mathcal{L}$ since the function $h:\mathcal{S}\rightarrow\mathbb{R}_{\geq0}$ given by $h(F)=h_{\mathcal{L}}(F(P))$ depends on $P$ and $\mathcal{L}$.   
\end{proof}
\section{Representing Semigroups by Endomorphisms}\label{sec:representations}
Having established asymptotic growth rates for abstract partially commutative semigroups, we now turn to the problem of constructing such semigoups using endomorphisms of varieties. In particular, we begin with the following general question: 
\begin{question}\label{quest:representation} Let $K$ be a field of characteristic zero and let $\mathcal{S}$ be a finitely generated, partially commutative semigroup. When does there exist a variety $V_{/K}$, an ample invertible sheaf $\mathcal{L}_{/K}$ on $V$, and endomorphisms $f_1,\dots, f_n$ of $V$ all defined over $K$ and all polarized with respect to $\mathcal{L}$ such that $\langle f_1,\dots,f_n\rangle\cong \mathcal{S}$.    
\end{question}
\begin{defin} We say that a graph $\mathcal{G}$ is \emph{representable} over $K$ if there exist $V$, $\mathcal{L}$, and a partially commutative semigroup $\mathcal{S}=\langle f_1,\dots,f_n\rangle$ as in Question \ref{quest:representation} such that $\mathcal{G}(\mathcal{S})=\mathcal{G}$.     
\end{defin}
As we will see, in dimension one the number of partially commutative semigroups is limited to free products of free commutative semigroups (equivalently, the only possible commutation graphs are disjoint unions of complete subgraphs). Nevertheless, in higher dimensions more possibilities open up. For example, we construct polarizable semigroups of endomorphisms of $V=E\times E\times E$ where $E$ is an elliptic curve that are partially commutative with commutation graph $P_3$, the line graph with $3$ vertices; see Proposition \ref{prop:abelian+examples} below. To do this, we construct semigroups in $\textup{Mat}_{3\times 3}(\mathbb{Z})$ with the desired properties (although the polarization conditions complicates things). On the other hand, the possible partially commutative semigroups in $\textup{Mat}_{n\times n}(\mathbb{Z})$ are also limited; for example,  in $\textup{Mat}_{2\times 2}(\mathbb{Z})$ any partially commutative semigroup is the direct product of a free commutative semigroup with a free product of free commutative semigroups; see \cite{MR1029251}. Nevertheless, the evidence so far suggests that the number of representable partially commutative semigroups should increase with the dimension of the underlying variety.
\subsection{A Ping-Pong Lemma for Subsemigroups}
To construct examples of partially commutative semigroups of endomorphisms, we use the following result, which is a straightforward adaption of the standard ping-pong lemma for elements; see, for instance, \cite[\S II.B]{de2000topics}. This may be well known to the experts, but without a suitable reference, we state and prove it here:
\begin{lem}[Ping-pong]\label{lem:ping+pong}
Let $\mathcal{S}$ be a semigroup, let $n\geq 2$, and let $\mathcal{S}_1,\dots,\mathcal{S}_n$ be subsemigroups of $\mathcal{S}$ that together generate $\mathcal{S}$, i.e., $\mathcal{S}=\langle \mathcal{S}_1\cup\dots\cup\mathcal{S}_n\rangle$. Suppose that $\mathcal{S}$ acts on the left on a set $X$, and that there are pairwise disjoint subsets $X_1,\dots,X_n$ of $X$ and a point $x_0\in X\setminus\bigcup_{i} X_i$ such that, setting $Y_i:=\{x_0\}\cup\bigcup_{j\neq i}X_j$, the following hold for all $1\leq i\leq n$:
\vspace{.15cm}
\begin{enumerate}
\item[\textup{(1)}] $\mathcal{S}_i\cdot Y_i\subseteq X_i$;
\vspace{.15cm}
\item[\textup{(2)}] the map $\mathcal{S}_i\times Y_i\to X$ given by $(s,y)\mapsto s\cdot y$ is injective.
\vspace{.15cm}
\end{enumerate}
Then the canonical map $\pi\colon\mathcal{S}_1*\dots*\mathcal{S}_n\to\mathcal{S}$ is a semigroup isomorphism.
\end{lem}

\begin{proof}
Surjectivity follows from the assumption that the $\mathcal{S}_i$ generate $\mathcal{S}$. For injectivity, recall that every element of $\mathcal{S}_1*\dots*\mathcal{S}_n$ is represented by a unique word $w=s_1s_2\cdots s_k$ with $k\geq 1$, $s_j\in\mathcal{S}_{i_j}$ for $1\leq j\leq k$, and $i_j\neq i_{j+1}$ for $1\leq j<k$; multiplication is concatenation followed by multiplying adjacent letters that lie in the same factor. We call this word the normal form of the element and $k$ its length. Since $\pi$ restricts to the inclusion on each factor, $\pi(w)=s_1s_2\cdots s_k$, the product now being taken in $\mathcal{S}$.

\underline{Step 1.} We claim that if $w=s_1\cdots s_k$ is in normal form with $s_1\in\mathcal{S}_{i_1}$, then $\pi(w)\cdot x_0\in X_{i_1}$. We argue by induction on $k$. If $k=1$, then $\pi(w)\cdot x_0=s_1\cdot x_0\in X_{i_1}$ by hypothesis (1), since $x_0\in Y_{i_1}$. If $k\geq 2$, write $w=s_1u$ with $u=s_2\cdots s_k$; no merging occurs in this product since $i_2\neq i_1$. The word $u$ is in normal form with first letter in $\mathcal{S}_{i_2}$, so $\pi(u)\cdot x_0\in X_{i_2}\subseteq Y_{i_1}$ by the induction hypothesis and the definition of $Y_{i_1}$. Hence
\begin{equation}\label{eq:ping+pong+1}
\pi(w)\cdot x_0=\big(s_1\,\pi(u)\big)\cdot x_0=s_1\cdot\big(\pi(u)\cdot x_0\big)\in X_{i_1}
\end{equation}
by hypothesis (1) again.

\underline{Step 2.} We show that the map $w\mapsto\pi(w)\cdot x_0$ is injective on $\mathcal{S}_1*\dots*\mathcal{S}_n$. Let $w=s_1\cdots s_k$ and $w'=s_1'\cdots s_\ell'$ be in normal form, with $s_j\in\mathcal{S}_{i_j}$ and $s_j'\in\mathcal{S}_{m_j}$, and suppose that $\pi(w)\cdot x_0=\pi(w')\cdot x_0$. We prove that $w=w'$ by strong induction on $\max\{k,\ell\}$. By Step 1, we have that $\pi(w)\cdot x_0\in X_{i_1}\cap X_{m_1}$. Thus, it must be the case that $i_1=m_1$, since the $X_i$ are pairwise disjoint. In particular, $s_1,s_1'\in\mathcal{S}_{i_1}$.

Now, if $\max\{k,\ell\}=1$, then $k=\ell=1$ and $s_1\cdot x_0=s_1'\cdot x_0$. In particular, since $x_0\in Y_{i_1}$, hypothesis (2) then gives $s_1=s_1'$, that is, $w=w'$.

Next suppose that $\max\{k,\ell\}>1$; by symmetry we may assume that $k>1$. Write $w=s_1u$ with $u=s_2\cdots s_k$. As in Step 1, we have that $\pi(u)\cdot x_0\in X_{i_2}\subseteq Y_{i_1}$ and $\pi(w)\cdot x_0=s_1\cdot(\pi(u)\cdot x_0)$.

\textbf{Case $\ell=1$.} Then we compute that
\[
s_1\cdot\big(\pi(u)\cdot x_0\big)=\pi(w)\cdot x_0=\pi(w')\cdot x_0=s_1'\cdot x_0.
\]
Both $\pi(u)\cdot x_0$ and $x_0$ lie in $Y_{i_1}$, so hypothesis (2) gives $\pi(u)\cdot x_0=x_0$. This is impossible, since $\pi(u)\cdot x_0\in X_{i_2}$ while $x_0\notin\bigcup_i X_i$. Hence this case cannot occur.

\textbf{Case $\ell>1$.} Write $w'=s_1'u'$ with $u'=s_2'\cdots s_\ell'$. The word $u'$ is in normal form with first letter in $\mathcal{S}_{m_2}$, and $m_2\neq m_1=i_1$, so $\pi(u')\cdot x_0\in X_{m_2}\subseteq Y_{i_1}$ by Step 1. Applying \eqref{eq:ping+pong+1} to both $w$ and $w'$ gives
\[
s_1\cdot\big(\pi(u)\cdot x_0\big)=\pi(w)\cdot x_0=\pi(w')\cdot x_0=s_1'\cdot\big(\pi(u')\cdot x_0\big),
\]
so hypothesis (2) yields $s_1=s_1'$ and $\pi(u)\cdot x_0=\pi(u')\cdot x_0$. Since $u$ and $u'$ are in normal form, of lengths $k-1$ and $\ell-1$ with $\max\{k-1,\ell-1\}<\max\{k,\ell\}$, the induction hypothesis implies that $u=u'$. Therefore $w=s_1u=s_1'u'=w'$.

\underline{Step 3.} If $\pi(w)=\pi(w')$, then $\pi(w)\cdot x_0=\pi(w')\cdot x_0$, and so $w=w'$ by Step 2. Hence $\pi$ is injective, and therefore an isomorphism, which completes the proof.
\end{proof}
\subsection{Free Products of Semigroups} 
The main result of this section is a classification of partially commutative semigroups of rational functions in one variable:    
\begin{thm}\label{thm:classification+dimension1} Let $K$ be a field of characteristic zero and let $\mathcal{S}$ be a finitely generated partially commutative semigroup. Then $\mathcal{S}$ can be represented by a semigroup of rational functions in $K(x)$ all of degree at least $2$ if and only if $\mathcal{S}$ is a free product of free commutative semigroups.     
\end{thm}
\begin{remark} Equivalently, $\mathcal{S}$ is represented by a semigroup of rational maps in $K(x)$ of degree at least $2$ if and only of $\mathcal{G}(\mathcal{S})=K_{m_1}\sqcup K_{m_2}\dots \sqcup K_{m_k}$ for some $k\geq1$ and $m_1,\dots, m_k\geq1$ In the literature, such graphs are known as cluster graphs.  
\end{remark}
To prove one direction of this result, we first note that if $\mathcal{S}$ is a partially commutative semigroup as in Theorem \ref{thm:classification+dimension1}, then the commutation graph $\mathcal{G}(\mathcal{S})$ cannot contain $P_3$ as a subgraph; here $P_3$ is the (line) graph with three vertices, $\{f_1,f_2,f_3\}$, such that there is an edge between $f_1$ and $f_2$ and $f_2$ and $f_3$, but no edge between $f_1$ and $f_3$.  
\begin{lem}\label{lem:no+P3's+dimension+one} Let $\mathcal{S}$ be a partially commutative semigoup of rational functions in $K(x)$ all  of degree at least $2$. Then $\mathcal{G}(\mathcal{S})$ does not contain the subgraph $P_3$.    
\end{lem}
\begin{proof} Suppose that $\mathcal{G}(\mathcal{S})$ contains a subgraph $P_3$ with vertices  corresponding to the functions $f_1,f_2,f_3$. Then $f_2$ commutes with both $f_1$ and $f_3$ and $\langle f_1,f_3 \rangle$ is free. More specifically, $\langle f_1,f_2\rangle$ and $\langle f_2,f_3 \rangle$ are free commutative semigroups of rank $2$ --- the only relations between generators are the commuting ones. Hence, $\langle f_1,f_2\rangle$ and $\langle f_2,f_3 \rangle$ have polynomial growth; in fact, both semigroups have quadratic growth. In particular, after embedding $K$ into $\mathbb{C}$, we may apply \cite[Corollary 4.12]{MR4780496} to both pairs of functions $(f_1,f_2)$ and $(f_2,f_3)$  separately to deduce that $\textup{PrePer}(f_1)=\textup{PrePer}(f_2)$ and $\textup{PrePer}(f_2)=\textup{PrePer}(f_3)$, where $\textup{PrePer}(f)$ denotes the set of complex preperiodic points of $f$. However, since $\textup{PrePer}(f_1)=\textup{PrePer}(f_3)$, we see again from \cite[Corollary 4.12]{MR4780496} that $\langle f_1,f_3\rangle$ has polynomial growth. But this contradicts our assumption that $\langle f_1,f_3\rangle$ is free, since free semigroups have exponential growth. The claim follows.   
\end{proof}
Next, we have the following elementary observation: if a finite graph does not contain $P_3$ as a subgraph, then it must be a disjoint union of complete subgraphs (or cliques). 
\begin{lem}\label{lem:graph+no+P3} Let $\mathcal{G}=(\mathcal{V},\mathcal{E})$ be a graph with finitely many vertices. If $G$ does not contain $P_3$ as a subgraph, then $G$ is a disjoint union of complete subgraphs.   
\end{lem}
\begin{proof} Let $C$ be a connected component of $\mathcal{G}$. Then it suffices to show that $C$ is a complete graph. Let $u,v\in \mathcal{V}$ be vertices of $C$. Since $C$ is connected, there is a shortest path $u=w_0\dots w_k=v$ with $(w_i,w_{i+1})\in \mathcal{E}$ from $u$ to $v$. If $k\geq2$, then $\{w_0,w_1\},$$\{w_1,w_2\}\in \mathcal{E}$ and $\{w_0,w_2\}\not\in \mathcal{E}$, since otherwise $w_0w_2\dots w_k$ is a shorter path from $u$ to $v$. But then the subgraph with vertices $w_0,w_1, w_2$ forms a $P_3$ in $\mathcal{G}$, a contradiction. Hence, $k=1$ and $C$ is complete.          
\end{proof}
Conversely, we can explicitly construct arbitrary free products of free commutative semigroups using polynomials in one variable:    
\begin{prop}\label{prop:Free+products+free+abelian}
Let $K$ be a number field and let $k,m_1,\dots, m_k\geq1$. Then choose distinct constants $b_1,\dots, b_k\in K$ and degrees $d_{ij}\geq2$ for $1\leq i\leq k$ and $1\leq j\leq m_i$. Moreover, assume that $\gcd(d_{ij},d_{i\ell})=1$ for all  $1\leq i\leq  k$ and $1\leq j<\ell\leq m_i$. Define the polynomials 
\[f_{ij}(x)=(x+b_i)^{d_{ij}}-b_i.\]
Then $\mathcal{S}:=\langle f_{ij}\rangle\cong \mathcal{A}(m_1)*\dots*\mathcal{A}(m_k)$ is partially commutative, where $\mathcal{A}(m)$ denotes the free commutative semigroup of rank $m$. Equivalently, $\mathcal{G}(\mathcal{S})=K_{m_1}\sqcup\dots \sqcup K_{m_k}$. 
\end{prop}
In light of the Ping-Pong Lemma \ref{lem:ping+pong}, we define the following sets 
\begin{equation}\label{eq:ping+pong+sets}
\begin{split}
X=\{f\in K[x]:\deg(f)\geq1\;\text{and}\;\text{$f$ is monic}\}\;\;\text{and}\;\; X_i=\{g^e-b_i\;:\: g\in X \;\text{and}\; e\geq2\}
\end{split} 
\end{equation}
for $1\leq i\leq k$. Moreover, it is clear that the map $\mathcal{S}\times X\rightarrow X$ given by $(F,P)\rightarrow F\circ P$ gives a left action of $\mathcal{S}$ on $X$. Moreover, we have the following additional properties:    
\begin{lem}\label{lem:power+maps} Let $f_{ij}$ be defined as in Proposition \ref{prop:Free+products+free+abelian}, let $\mathcal{S}_i=\langle f_{i1},\dots, f_{im_i} \rangle$ for $1\leq i\leq k$, and let $L_i(x)=x+b_i$. Moreover, let $X$ and $X_i$ be as defined on \eqref{eq:ping+pong+sets}. Then the following hold: 
\vspace{.05cm}
\begin{enumerate}
\item[\textup{(1)}] $\mathcal{S}_i\cdot X\subseteq X_i$ for all $1\leq i\leq k$ and $X_i\cap X_j=\varnothing$ for all $i\neq j$.  
\vspace{.15cm}
\item[\textup{(2)}] For all $G,G'\in \mathcal{S}_i$, we have that $G=G'$ if and only if $\deg(G)=\deg(G')$.    
\vspace{.15cm} 
\item[\textup{(3)}] If $P,Q\in X$ and $G\circ P=G\circ Q$ for some $G\in \mathcal{S}_i$, then $P=Q$.
\vspace{.15cm} 
\item[\textup{(4)}] If $P,Q\in X$ and $G\circ P=G'\circ Q$ for some $G,G'\in \mathcal{S}_i$ with $G\neq G'$, then either $P\in X_i$ or $Q\in X_i$.   
\end{enumerate}
\end{lem}
\begin{proof} Let $1\leq i\leq k$ and $1\leq j\leq m_i$. If $P\in X$, then $g:=P+b_{i}$ is monic, non-constant and $f_{ij}(P)=(P+b_i)^{d_{ij}}-b_i=g^{d_{ij}}-b\in X_i$ by definition of $X_i$. Moreover, since this statement holds for all $1\leq j\leq m_i$, it follows that $\mathcal{S}_i\cdot X\subseteq X_i$ for all $1\leq i\leq k$. Likewise, the claim that the $X_i$ are disjoint follows from the fact that the Fermat-Catalan equation $A^{e_1}-B^{e_2}=c$ for $e_1,e_2\geq2$ and non-zero $c\in K$ has no solutions $A,B\in K[x]$. This is a simple consequence of the Mason ABC theorem for the polynomial ring $K[x]$; see, for instance, \cite[Lemma 3.2]{hindes2022orbit}. Here we use also that $b_i\neq b_j$ for $i\neq j$ by assumption. 

For statement (2), note that $f_{ij}=L_{b_i}^{-1}\circ x^{d_{ij}}\circ L_{b_i}$, so that each $f_{ij}$ is power map conjugated by the linear transformation $L_{b_i}$. Hence, $f_{ij}$ and $f_{i\ell}$ commute for all $1\leq j,\ell\leq m_i$ since arbitrary power maps commute. Thus, every element $G\in \mathcal{S}_i$ is of the form 
\begin{equation}\label{eq:poly+decomp}
G=f_{i1}^{e_1(G)}\circ\dots\circ f_{im_i}^{e_{m_i(G)}}
\end{equation}
for some $e_1(G),\dots e_{m_i}(G)\geq0$. In particular, if $G,G'\in \mathcal{S}_i$ satisfy $\deg(G)=\deg(G')$, then 
\[d_{i1}^{e_1(G)}\cdots d_{im_i}^{e_{m_i}(G)}=d_{i1}^{e_1(G')}\cdots d_{im_i}^{e_{m_i}(G')}.\]
However, since the integers $d_{i1},\dots,d_{im_i}$ are pairwise coprime,  
they are multiplicatively independent. Hence, $e_j(G)=e_j(G')$ for all $1\leq j\leq m_i$, from which the claim follows.   

For statement (3) we may write $G=L_{b_i}^{-1}\circ x^n\circ L_{b_i}$ for some $n\geq2$, so that if $G\circ P=G\circ Q$ for some $P,Q\in X$, then writing $P'=L_{b_i}\circ P$ and $Q'=L_{b_i}\circ Q$, we see that    
\begin{equation}\label{eq:power+overlap+1}
(P')^n=x^n\circ L_{b_i}\circ P=L_{b_i}\circ L_{b_i}^{-1}\circ x^n\circ L_{b_i}\circ P=L_{b_i}\circ G\circ P=L_{b_i}\circ G\circ Q=(Q')^n.
\end{equation}
Hence, $(P')^n=(Q')^n$ for some $n\geq 2$. Hence, $P'$ and $Q'$ must be scalar multiples by an $n$th root of unity $\zeta\in K$. But, since $P$ and $Q$ are assumed to be monic, so are $P'$ and $Q'$. Therefore, we see that $\zeta=1$ and $P'=Q'$. But then, $P=L_b^{-1}\circ P'=L_b^{-1}\circ Q'=Q$ as claimed.  

Finally, suppose that $G\circ P=G'\circ Q$ for some $G,G' \in \mathcal{S}_i$ and some $P,Q\in X$. Then, as in \eqref{eq:poly+decomp}, we may write $G=f_{i1}^{r_1}\circ\dots \circ f_{im_i}^{r_{m_i}}$ and $G'=f_{i1}^{s_1}\circ\dots \circ f_{im_i}^{s_{m_i}}$ for some unique integers $r_i,s_j\geq0$. Note also that by statement (3), the action of $\mathcal{S}_i$ on $X$ is left-cancellative. Hence, if $r_j\neq s_j$ for some $j$, without loss, say $r_j=s_j+t_j$ for $t_j>0$, then we may move the factors $f_{ij}^{r_j}$ and $f_{ij}^{s_{j}}$ to the far left side of the expression $G\circ P=G'\circ Q$ and cancel $f_{ij}^{s_j}$ from both sides of this equation. In particular, we can eliminate the compositional factors of $f_{ij}$ on the right hand side. Continuing in this way (and after relabeling the indices if necessary), we see that 
\begin{equation}\label{eq:power+overlap+2}
\begin{split} 
L_{b_i}^{-1}\circ x^{d_{i1}^{t_1}\cdots d_{ip}^{t_p}}\circ L_{b_i}\circ P &=f_{i1}^{t_1}\circ\dots\circ f_{ip}^{t_p}\circ P\\[5pt] 
&=f_{i(p+1)}^{t_{p+1}}\circ\dots\circ f_{i(p+q)}^{t_{(p+q)}}\circ Q\\[5pt]
&=L_{b_i}^{-1}\circ x^{d_{i(p+1)}^{t_{(p+1)}}\cdots d_{i(p+q)}^{t_{(p+q)}}}\circ L_{b_i}\circ Q,
\end{split}
\end{equation}
for some $p,q\geq0$ and $t_j>0$. In particular, if $G\neq G'$, then it must be the case that either $p$ or $q$ is strictly positive. Without loss, suppose that $p>0$. Note that if $q=0$ (so that all of the power maps in the middle of the equation on the bottom of \eqref{eq:power+overlap+2} disappear), then setting $e=d_{i1}^{t_1}\cdots d_{ip}^{t_p}$ and $P'=L_{b_i}\circ P$ we see that 
\[(P')^e-b_i=L_b^{-1}\circ x^e\circ L_b\circ P=Q\] 
for some $e\geq2$. Thus, $Q\in X_i$ as desired. Hence, we may assume that $p,q>0$. But in this case, if we set $e=d_{i1}^{t_1}\cdots d_{ip}^{t_p}$, set $P'=L_{b_i}\circ P$, set $Q'=L_{b_i}\circ Q$, and set $d=d_{i(p+1)}^{t_{(p+1)}}\cdots d_{i(p+q)}^{t_{(p+q)}}$, then we see by applying $L_{b_i}$ to both sides of \eqref{eq:power+overlap+2} that 
\[(P')^e=(Q')^d\] 
for some \textbf{\emph{coprime}} $e,d\geq2$; here we use the fact that the original degrees $d_{i1},\dots, d_{im_i}$ are all pairwise coprime). But then we choose integers $a,b$ such that $ae+bd=1$ and so compute that 
\[
\scalemath{.95}{
Q'=(Q')^{ae+bd}=((Q')^a)^e\cdot((Q')^d)^b=((Q')^a)^e\cdot ((P')^e)^b=((Q')^a)^e\cdot ((P')^b)^e=((Q')^a\cdot (P')^b)^e
}
\]
Therefore, $Q'$ is an $e$-th power of an element of $K(x)$. In fact, since $K[x]$ is integrally closed in $K(x)$, it must be the case that $Q'$ is an $e$-th power in $K[x]$. Write $Q'=g^e$ for some $g\in K[x]$. Then we have that
\[
Q=L_{b_i}^{-1}\circ L_{b_i}\circ Q=L_{b_i}^{-1}\circ Q'=L_{b_i}^{-1}\circ g^e=g^e-b_i,
\]
so that $Q\in X_i$ as claimed. This completes the proof of the lemma. 
\end{proof}
With the above result in place, we can now use the Ping-Pong Lemma \ref{lem:ping+pong} to produce arbitrary free products of free abelian semigroups using polynomials. 
\begin{proof}[(Proof of Proposition \ref{prop:Free+products+free+abelian})]
Let $f_{ij}$ be as in Proposition \ref{prop:Free+products+free+abelian} and let $X$ and $X_i$ be as in \eqref{eq:ping+pong+sets}. Moreover, for the basepoint $x_0$ in Lemma \ref{lem:ping+pong}, take $x_0=x$ to be the identity polynomial and let $Y_i:=\{x_0\}\cup\bigcup_{j\neq i} X_j$ be as in in Lemma \ref{lem:ping+pong}. Then $X$, $X_i$, and $x_0$ satisfy the setup of Lemma \ref{lem:ping+pong} by statement (1) of Lemma \ref{lem:power+maps}. Moreover, property (1) of Lemma \ref{lem:ping+pong} also holds by statement (1) of Lemma \ref{lem:power+maps}. As for property (2) of Lemma \ref{lem:ping+pong}, suppose that $G\cdot P=G\circ P=G'\circ Q=G'\circ Q$ for some $G,G'\in \mathcal{S}_i$ and some $P,Q\in Y_i$. If $G\neq G'$, then statement (4) of Lemma \ref{lem:power+maps} implies that either $P$ or $Q$ is $X_i$. Without loss, say $P\in X_i$. Then, by definition of $Y_i$, we have two cases: (1) $P\in X_j$ for some $j\neq i$ or (2) $P=x$. However, in case (1) we see that $P\in X_i\cap X_j$, a contradiction. Likewise, if $P=x$, then $x\in X_i$, which is clearly false (for instance, every element in $X_i$ has degree at least 2). Therefore, it must be the case that $G=G'$. But then $G\circ P=G'\circ Q=G\circ Q$, so that statement (3) of Lemma \ref{lem:power+maps} implies that $P=Q$. Therefore, property (2) of Lemma \ref{lem:ping+pong} holds. Hence, $\mathcal{S}\cong \mathcal{S}_1*\dots * \mathcal{S}_k$ follows from Lemma \ref{lem:ping+pong}.  
Finally, statement (2) of Lemma \ref{lem:power+maps} implies that the degree map $\deg: \mathcal{S}_i\rightarrow\langle d_{i1},\dots,d_{im_i}\rangle$ is a semigroup isomorphism. Moreover, $\langle d_{i1},\dots,d_{im_i}\rangle\cong \mathcal{A}(m_i)$ since the $d_{i1},\dots,d_{im_i}$ are multiplicatively independent. The result follows.  
\end{proof}
Finally, we have all of the tools in place to prove the main result of this section: a partially commutative semigroup $\mathcal{S}$ can be  represented by a semigroup of non-linear rational functions in $K(x)$ under composition if and only if it is a free product of free commutative semigroups.
\begin{proof}[(Proof of Theorem \ref{thm:classification+dimension1})]
Suppose that $\mathcal{S}$ is a partially commutative semigroup contained in $K(x)$ of maps of degree at least $2$. Then $\mathcal{G}(\mathcal{S})$ does not contain $P_3$ as a subgraph by Lemma \ref{lem:no+P3's+dimension+one}. Hence, $\mathcal{G}(\mathcal{S})$ is a disjoint union of complete subgraphs by Lemma \ref{lem:graph+no+P3}. In particular, $\mathcal{S}$ is a free product of free commutative semigroups, each factor corresponding to a connected component of $\mathcal{G}(S)$. Conversely, if $\mathcal{S}$ is a finitely generated, abstract semigroup that is a free product of free commutative semigroups, then $\mathcal{S}$ is isomorphic to an explicit semigroup in $K[x]$ of polynomials of degree at least $2$ given in Proposition \ref{prop:Free+products+free+abelian}. 
\end{proof}
\subsection{Direct Products of Semigroups} 
In this section, we use products of elliptic curves to produce different families of partially commutative semigroups. In particular, when $V$ is an abelian variety, we say that a semigroup $\mathcal{S}$ of self-maps of $V$ is \emph{polarized} with respect to an ample (Weil) divisor $D$ on $V$ if
$F^*(D)\sim d(F)\,D$ for every $F\in \mathcal{S}$ and some integer $d(F)\geq 2$. 
\begin{remark}
This agrees with the definition (given earlier) in terms of an ample invertible sheaf $\mathcal{L}$ with $f^*(\mathcal{L})\cong \mathcal{L}^{\otimes d(f)}$. Indeed, abelian varieties are nonsingular
\cite[Ch.~I, \S1]{MilneAV}, hence locally factorial \cite[Theorem ~20.3]{Matsumura}.
On a locally factorial variety, Weil and Cartier divisors coincide, and $D\mapsto\mathcal{O}_V(D)$ identifies divisors modulo linear equivalence with invertible sheaves modulo isomorphism \cite[Ch.~II, Prop.~6.11 and Cor.~6.16]{Hartshorne}.
Moreover, this identification is compatible with pullbacks.
\end{remark}
In particular, we produce \emph{direct products} of free semigroups with free commutative semigroups. For example, the graph $P_3$ is representable in dimension three.
\begin{prop}\label{prop:abelian+examples} Let $K$ be a number field, let $E$ be an elliptic curve defined over $K$, and let $V=E\times E\times E$. Define the divisor $D=2(\pi_1^*(\mathcal{O})+\pi_2^*(\mathcal{O})+\pi_3^*(\mathcal{O}))$ on $K$, where $\pi_i:V\rightarrow E$ is the projection map onto the $i$th component and $\mathcal{O}$ is the identity of $E$. Now fix a non-trivial pythagorean triple $a,b,c\in\mathbb{Z}_{\geq2}$ and an integer $k>1$ coprime to $c$. Define the endomorphisms $\alpha$ and $\beta$ on $V$ given by the following matrices:
\vspace{.1cm} 
\[\alpha=
\begin{bmatrix}
     a&-b &0 \\
     b&\;a&0 \\
     0&\;0&c 
\end{bmatrix}
,
\;\;\;\;
\beta=
\begin{bmatrix}
     kc&0&\;0\\
     0&ka&-kb\\
     0&kb&\;ka 
\end{bmatrix}
\vspace{.1cm} 
.\]
Then the following statements hold:\vspace{.1cm}  
\begin{enumerate}
\item[\textup{(1)}] The semigroup $\langle \alpha,\beta\rangle$ is free and polarized with respect to $D$.
\vspace{.15cm} 
\item[\textup{(2)}] Fix $n\geq1$ and let $f_i=\alpha\circ \beta^i$ for $1\leq i\leq n$. Then $\langle f_1,\dots, f_n\rangle$ is a free semigroup that is polarized with respect to $D$.\vspace{.15cm} 
\item[\textup{(3)}] Fix $m\geq1$, let $q_1,\dots,q_m>kc$ be prime numbers, and let $g_j=[q_j]$ be the multiplication by $q_i$ map on $V$ for $1\leq j\leq m$. Then $\langle g_1,\dots,g_m\rangle$ is a free commutative semigroup that is polarized with respect to $D$.\vspace{.15cm}
\item[\textup{(4)}] Define $\mathcal{S}(a,b,c,k,n,m):=\langle f_1,\dots, f_n,g_1\dots,g_m \rangle$. Then  
\[\mathcal{S}(a,b,c,k,n,m)\cong\mathcal{F}(n)\times\mathcal{A}(m),\]
and $\mathcal{S}(a,b,c,k,n,m)$ is polarized with respect to $D$. Hence, $\mathcal{S}(a,b,c,k,n,m)$ is a partially commutative semigroup with commutation graph $E_n\vee K_m$.\vspace{.1cm}  
\end{enumerate}
In particular, the graph $E_n\vee K_m$ is representable over $K$ for all $n,m\geq0$.   
\end{prop}

We use the following polarization criterion for certain maps of powers of elliptic curves with respect to the product divisor $D$ above.  
\begin{lem}\label{lem:polarized+split+maps+abelian+varieties} Let $V=E^n$ for some elliptic curve $E$ and $n\geq1$ and let $D=2\sum_{i=1}^n\pi_i^*(\mathcal{O})$ where $\pi_i:V\rightarrow E$ is the projection map onto the $i$th component. If $f$ is an endomorphism of corresponding to a matrix $M\in \textup{Mat}_{n\times n}(\mathbb{Z})$ and $M^TM=dI$ for some $d\in\mathbb{Z}$ where $I$ is the identity matrix, then $f^*(D)\sim d D$.  \end{lem}
\begin{proof} Let $g_i:E^n\rightarrow E$ be given by $g_i=\pi_i\circ f$ so that
\begin{equation}\label{eq:polarized+split+maps+abelian+varieties1}
f^*(D)=f^*\Big(2\sum_{i=1}^n\pi_i^*(\mathcal{O})\Big)\sim2 \sum_{i=1}^nf^*(\pi_i^*(\mathcal{O}))\sim \sum_{i=1}^n2g_i^*(\mathcal{O}).
\end{equation} 
It therefore suffices to compute $2g_i^*(\mathcal{O})$ to compute $f^*(D)$. Now write $f$ in coordinates 
\[f(P_1,\dots,P_n)=\Big(\sum_{j=1}a_{1j}P_j,\dots, \sum_{j=1}a_{nj}P_j\Big)\] 
for some $P_1,\dots, P_n\in E$ and some $a_{ij}\in\mathbb{Z}$. In particular, it is clear that $g_i=\sum_{j=1}^n a_{ij}\pi_j$ for all $1\leq i\leq n$. Moreover, we may associate $f$ with the matrix $M=(a_{ij})\in \textup{Mat}_{n\times n}(\mathbb{Z})$ and assume that $M^T M=d I$ for some $d$; here $M^T$ denotes the transpose of $M$. Hence, we assume that  
\begin{equation}\label{eq:matrix+condition} 
(\textup{I})\;\; d=\sum_{i=1}^na_{ij}^2 \;\;\;\; \text{and}\;\;\;\;(\textup{II})\;\; 0=\sum_{i=1}a_{ij}a_{ik}
\end{equation}
for all $1\leq j\leq n$ for (I) and all $1\leq j\neq k\leq n$ for (II). Next, it follows from the Theorem of the Cube \cite[Theorem A.7.2.1]{hindry2013diophantine} that the pairing \vspace{.1cm}
\begin{equation}\label{eq:thm+of+cube}
B:\,\Hom(E^n,E)\times\Hom(E^n,E)\rightarrow \Pic(E^n),\;\;\;\;\;B(\varphi,\psi)=(\varphi+\psi)^*(\mathcal{O})-\varphi^*(\mathcal{O})- \psi^*(\mathcal{O}) 
\vspace{.1cm} 
\end{equation}
is symmetric and bilinear; here (unlike the results in \cite[\S A.7.2]{hindry2013diophantine} which we will use), by $\varphi+\psi$ for $\varphi,\psi\in \Hom(E^n,E)$ we mean the map $\varphi+\psi: E^n\rightarrow E$ given by $(\varphi+\psi)(x)=\varphi(x)+\psi(x)$. In particular, $B$ as defined in \eqref{eq:thm+of+cube} is clearly symmetric.     

We grant this bilinearity claim for now and complete the proof. In particular, the associated map $Q:\Hom(E^n,E)\rightarrow\Pic(E^n)$ given by $Q(\varphi)=B(\varphi,\varphi)$ is a quadratic form, and
\vspace{.05cm} 
\begin{equation}\label{eq:polarized+split+maps+abelian+varieties2}
\begin{split} 
Q(\varphi)&=B(\varphi,\varphi)\sim([2]\circ \varphi)^*(\mathcal{O})-2\varphi*(\mathcal{O})\sim\varphi^*([2]^*\mathcal{O})-2\varphi^*(\mathcal{O})\\[5pt]
&\sim\varphi^*(4\mathcal{O})-2\varphi^*(\mathcal{O})\sim4\varphi^*(\mathcal{O})-2\varphi^*(\mathcal{O})\sim2\varphi^*(\mathcal{O})\\[4pt] 
\end{split}
\end{equation}
holds for all $\varphi\in\Hom(E^n,E)$. Hence, we compute that  
\begin{equation}\label{eq:polarized+split+maps+abelian+varieties3}
\begin{split}
2g_i^*(\mathcal{O})&\sim Q(g_i)=Q\Big(\sum_{j=1}^n a_{ij}\pi_j\Big)\sim \sum_{j=1}^n a_{ij}^2 Q(\pi_j)+2\sum_{k<j}a_{ij}a_{ik}B(\pi_j,\pi_k)\\[5pt]
&\sim 2\sum_{j=1}^n a_{ij}^2 \pi_j^*(\mathcal{O})+2\sum_{k<j}a_{ij}a_{ik}B(\pi_j,\pi_k)
\end{split}
\end{equation}
holds for any fixed $1\leq i\leq n$. In particular, combining \eqref{eq:polarized+split+maps+abelian+varieties1}, \eqref{eq:matrix+condition}, and \eqref{eq:polarized+split+maps+abelian+varieties3}, we deduce that
\begin{equation*}
\begin{split}
f^*(D)&\sim \sum_{i=1}^n2g_i^*(\mathcal{O})
\sim2\sum_{i=1}^n\bigg(\sum_{j=1}^n a_{ij}^2 \pi_j^*(\mathcal{O})\bigg)+2\sum_{i=1}^n\bigg(\sum_{k<j}a_{ij}a_{ik}B(\pi_j,\pi_k)\bigg)\\[5pt]
&\sim 2\sum_{j=1}^n\bigg(\sum_{i=1}^n a_{ij}^2 \pi_j^*(\mathcal{O})\bigg)
+2\sum_{k<j}\bigg(\sum_{i=1}^n a_{ij}a_{ik}B(\pi_j,\pi_k)\bigg)\\[5pt]
&\sim 2\sum_{j=1}^n\bigg(\sum_{i=1}^n a_{ij}^2\bigg)\pi_j^*(\mathcal{O})
+2\sum_{k<j}\bigg(\sum_{i=1}^n a_{ij}a_{ik}\bigg)B(\pi_j,\pi_k)\\[5pt]
&\sim 2\sum_{j=1}^n d\pi_j^*(\mathcal{O})\sim d\Big(2\sum_{j=1}^n\pi_j^*(\mathcal{O})\Big)=d D  
\end{split}
\end{equation*}
as claimed. Finally, we return to the bilinearity of $B$. For simplicity, we fix some notation. Let $V=E^n$, let $c$ be any divisor class of $E$, let $a:E\times E\rightarrow E$ be the addition map $a(P,Q)=P+Q$, and let $\Delta: V\rightarrow V\times V$ be the diagonal map $\Delta(x)=(x,x)$. Likewise, let $p_1,p_2: V\times V\rightarrow V$ be the projections onto the first and second coordinates, respectively. Then \cite[Corollary A.7.2.4]{hindry2013diophantine} implies that the pairing $J:\Hom(V,E)\times \Hom(V,E)\rightarrow \Pic(V\times V)$ given by 
\begin{equation}\label{eq:other+pairing}
J(\varphi,\psi)=(a\circ (\varphi,\psi))^*c-p_1^*(\varphi^*c)-p_2^*(\psi^*c)
\end{equation}
is bilinear, where $(\varphi,\psi):V\times V\rightarrow E\times E$ is the map $(\varphi,\psi)(x,y)=(\varphi(x),\psi(y))$; see the paragraph above Corollary A.7.2.5 in \cite{hindry2013diophantine}. In particular, the pairing $B': \Hom(V,E)\times \Hom(V,E)\rightarrow \Pic(V)$ given by $B'(\varphi,\psi)=\Delta^*(J(\varphi,\psi))$ is bilinear also. In fact, by taking $c=[\mathcal{O}]$, we will see that $B(\varphi,\psi)=B'(\varphi,\psi)$ for all $\varphi,\psi\in \Hom(V,E)$, from which the claim follows. To see this, simply apply $\Delta^*$ to \eqref{eq:other+pairing}, use that $\Delta^*$ is linear, that $\Delta^*p_i^*=(p_i\circ\Delta)^*=\textup{Id}_V^*$, and that $a\circ (f,g)\circ \Delta=f+g$.         
\end{proof}
\begin{remark}\label{rmk:pyth} Note that to any pythagorean triple $a^2+b^2=c^2$ for some $a,b,c\in\mathbb{Z}$, we can construct two matrices, 
\[\alpha(a,b,c)=
\begin{bmatrix}
     a&-b &0 \\
     b&\;a&0 \\
     0&\;0&c 
\end{bmatrix}
\;\;\;\;
\text{and}
\;\;\;\; 
\beta(a,b,c)=
\begin{bmatrix}
     c&0&\;0\\
     0&a&-b\\
     0&b&\;a 
\end{bmatrix},
\vspace{.1cm} 
\]
corresponding to polarized endomorphisms of $V=E^3$ with respect to the same divisor $D$. 
\end{remark}
In particular, we can produce many polarized semigroups with respect to the divisor $D$; for instance by aplying Euclid's formula $(m,n)\rightarrow(a,b,c):=(m^2-n^2,2mn,m^2+n^2)$ or by scaling known triples. Moreover, it is likely that many of the semigroups produced in this way are free on the given basis. However, proving this fact can be very difficult (if not undecidable). On the other hand, by scaling $\alpha$ and $\beta$ in Remark \ref{rmk:pyth} by $1/c$, we obtain matrices in the special orthogonal group $\textup{SO}_3(\mathbb{R})$, where additional geometric techniques are available to prove freeness. With this in mind, we have the following simplified version of the main result in \cite{swierczkowski1994class}.     
\begin{thm}\label{thm:free+rotations}
Let $a,b,c$ be integers satisfying $a^2+b^2=c^2$. Define the matrices 
\[U_1(a,b,c)=
\begin{bmatrix}
     \frac{a}{c}&\frac{-b}{c} &0 \\
     \frac{b}{c}&\;\frac{a}{c}&0 \\
     0&\;0&1 
\end{bmatrix}
\;\;\;
\text{and} 
\;\;\;\;
U_2(a,b,c)=
\begin{bmatrix}
     1&0&\;0\\
     0&\frac{a}{c}&\frac{-b}{c}\\
     0&\frac{b}{c}&\;\frac{a}{c} 
\end{bmatrix}
\vspace{.1cm} 
.\]
Then the subgroup of $\textup{SO}_3(\mathbb{R})$ generated by $U_1$ and $U_2$ is free with free generators $U_1$ and $U_2$ if and only if $a/c\notin\{0,\pm{1},\pm{1/2}\}$.    
\end{thm}
\begin{remark}\label{rmk:free+group+implies+semigroup} We note that since $U_1$ and $U_2$ generate a free group $G$, the semigroup $\langle U_1,U_2 \rangle$ they generate is free. This follows immediately from the Normal Form Theorem for free groups; see \cite[Theorem 4, p.5]{Cohen1989}: every string $w_1,\dots w_n$ of elements $w_i\in\{U_1,U_2\}$ contains no inverses and so represents a \emph{reduced word} when viewed of as an element of the free group $G$. Hence, by the uniqueness of reduced words in $G$, the subsemigroup $\langle U_1,U_2 \rangle$ of $G$ is free.    
\end{remark}
In particular, it follows from Theorem \ref{thm:free+rotations} that many of the semigroups generated by $\alpha(a,b,c)$ and $\beta(a,b,c)$ in Remark \ref{rmk:pyth} are free. However, freeness is not the only property we desire: in order to produce dirichlet series generating functions with well-behaved poles, we need endomorphisms with multiplicatively independent degrees, which $\alpha(a,b,c)$ and $\beta(a,b,c)$ do not furnish. Nevertheless, scaling by a suitable integer $k$ coprime to $c$, we can obtain two endomorphisms, corresponding to $\alpha(a,b,c)$ and $\beta(ak,bk,ck)$, with multiplicatively independent degrees. Moreover, when done carefully, scaling preserves freeness; in what follows, $\SL_n(K)$ denotes the matrices of determinant one over a field $K$.  
\begin{lem}\label{lem:det+trick} Let $U_1,\dots, U_r\in \SL_n(K)$ and $c_1,\dots,c_r\in K^\times$ for some $r\geq1$ and $n\geq2$. Moreover, assume that $\langle U_1,\dots, U_r \rangle$ is a free semigroup and that $c_1,\dots, c_r$ are multiplicatively independent. Then $\langle c_1U_1,\dots, c_rU_r \rangle $ is a free semigroup.    
\end{lem}
\begin{proof} Let $U_1,\dots, U_r$ and $c_1,\dots, c_r$ be as in the statement and define $A_i=c_i U_i$ for $1\leq i\leq r$. Now suppose that
\begin{equation}\label{eq:word+equality}
A_{i_1}\cdots A_{i_{s}}=A_{j_1}\cdots A_{j_{t}}.
\end{equation}
Then, substituting $A_{i_p}=c_{i_p}U_{i_p}$ and moving the constants to the left, we deduce that 
\begin{equation}\label{eq:det+trick}
c_1^{e_1}\cdots c_r^{e_{r}}(U_{i_1}\cdots U_{i_{s}})=c_1^{h_1}\cdots c_r^{h_r}(U_{j_1}\cdots U_{j_{t}}),
\end{equation}
where $e_i=\#\{i_k:i_k=i\}$ and $h_j=\#\{j_k:j_k=j\}$. Then, by taking determinants, we see 
\[c_1^{ne_1}\cdots c_r^{ne_{r}}=(c_1^{e_1}\cdots c_r^{e_{r}})^n=(c_1^{h_1}\cdots c_r^{h_r})^n=c_1^{nh_1}\cdots c_r^{nh_r}.\]
Moreover, by our assumption that the $c$'s are multiplicatively independent, we deduce that $ne_k=n h_k$ for all $1\leq k\leq r$. Thus $e_k=h_k$ for all  $1\leq k\leq r$. But then \eqref{eq:det+trick} implies that $U_{i_1}\cdots U_{i_{s}}=U_{j_1}\cdots U_{j_{t}}$, so our assumption that $\langle U_1,\dots, U_r \rangle$ is free implies that $s=t$ and $i_p=j_p$ for all $1\leq p\leq s$. Thus, \eqref{eq:word+equality} is a trivial relation as desired.         
\end{proof}
Finally, we note the following trick for constructing free semigroups of arbitrary rank within free semigroups of rank two.  
\begin{lem}\label{lem:free+trick} If $\langle A, B\rangle$ is free and $n\geq2$, then $\langle AB,AB^2,\dots,AB^n \rangle$ is a free.    
\end{lem}
\begin{proof} Suppose that $\theta_r\cdots \theta_1=\tau_s\cdots \tau_1$ for some $\theta_i,\tau_j\in\{AB,AB^2,\dots,AB^n\}$ and write $\theta_r=AB^{i}$ and $\tau_s=AB^{j}$. Note that if $\max\{r,s\}=1$, then $AB^{i}=AB^{j}$ so that $i=j$ by the freeness of $\langle A, B\rangle$. Therefore, we can assume that either $r$ or $s$ is at least $2$. Without loss, assume that $r\geq2$ so that 
\[AB^{i}Aw_1=AB^j w_2\]
for some $w_1\in\langle A,B\rangle$ and $w_2\in \langle A,B\rangle\cup\{\textup{I}\}$. But again by the freeness of $\langle A, B\rangle$, it must be the case that $i=j$ and $w_2\neq I$. Thus, $\theta_r=\tau_s$ and $s\geq2$. On the other hand, since free semigroups are left cancellative, we deduce that $\theta_{r-1}\cdots \theta_1=\tau_{s-1}\cdots \tau_1$. In particular, continuing in this way inductively, we see that $\langle AB,AB^2,\dots,AB^n \rangle$ is a free semigroup, as claimed.      
\end{proof}
We now have the tools in place to prove that $E_n\vee K_m$ is a representable graph over any number field. 
\begin{proof}[(Proof of Proposition \ref{prop:abelian+examples})]         Let $E$ be an elliptic curve defined over $K$, let $V=E^3$, let $(a,b,c)$ be a non-trivial pythagorean triple, and let $k\geq2$ be an integer satisfying $\gcd(c,k)=1$. Define endomorphisms $\alpha(a,b,c)$ and $\beta(ka,kb,kc)$ using the notation from Remark \ref{rmk:pyth}. In particular, Lemma \ref{lem:polarized+split+maps+abelian+varieties} implies that $\alpha(a,b,c)$ and $\beta(ka,kb,kc)$ are both polarized with respect to the divisor $D$. Moreover, these endomorphisms are both defined over $K$. Note next, in the notation from Theorem \ref{thm:free+rotations}, that $c\, U_1(a,b,c)=\alpha(a,b,c)$, that $kc\, U_2(a,b,c)=\beta(ka,kb,kc)$, and that $\big\langle U_1(a,b,c),U_2(a,b,c)\big\rangle$ is a free semigroup; see also Remark \ref{rmk:free+group+implies+semigroup}. Moreover, since $\gcd(k,c)=1$ the scalars $c$ and $ck$ are multiplicaitvely independent. Therefore, Lemma \ref{lem:det+trick} implies that $\big\langle \alpha(a,b,c),\beta(ka,kb,kc)\big \rangle$ is a free semigroup, which completes the proof of statement (1). However, for any $n\geq2$ it follows from Lemma \ref{lem:free+trick} that $\langle f_1,\dots,f_n\rangle$ is free semigroup, where $f_i=\alpha(a,b,c)\beta(ak,bk,ck)^i$ for $1\leq i\leq n$. Moreover, these endomorphisms are also polarized with respect to $D$, since $\alpha$ and $\beta$ are. Thus we have proved statement (2). Now let $q_1,\dots, q_m$ be distinct primes greater than $ck$. First, it is clear that the maps $g_i=[q_i]$, corresponding to the multiplication by $q_i$ maps on $V$, are polarized with respect to $D$ and central, since they correspond to the matrices of the form $q_i I$. Moreover, it is straightforward to check that the determinant map $\det: \langle q_1I,\dots,q_m I\rangle\rightarrow \langle q_1^3,\dots, q_m^3\rangle$ is a semigroup isomorphism; only injectivity requires an justification, but this follows immediately from the multiplicative independence of the $q$'s (and so also of their cubes) and the fact that these matrices commute (so any word is uniquely characterized by the multiplicity of each generator). Thus, we have established statement (3). Finally, to prove statement (4), it suffices to prove that the map
\begin{equation}\label{eq:map+decomp}
\langle f_1,\dots,f_n \rangle\times \langle g_1,\dots,g_m\rangle\rightarrow \langle f_1,\dots,f_n,g_1,\dots,g_m\rangle , \;\;\text{by}\;\;(F,G)\rightarrow F\circ G
\end{equation}
is a semigroup isomorphism. Surjectivity is clear: since the $g_i$'s are central, we may move any compositional factor of the $g$'s to the far right side of the expression. On the other hand, suppose that $F_1\circ G_1=F_2\circ G_2$ for some $F_1,F_2\in \langle f_1,\dots,f_n \rangle$ and $G_1,G_2\in \langle g_1,\dots,g_m \rangle$. Let $e_{ij}$ be the number of appearances of $g_i$ in the word $G_j$. Then taking determinants of the corresponding matrices gives: 
\[a_1 q_1^{3e_{11}}\cdots q_m^{3e_{m1}}=a_2 q_1^{3e_{12}}\cdots q_m^{3e_{m2}}\]
for some $a_1,a_2\in\langle c,ck\rangle$. But by construction, the $q$'s are primes strictly greater than $ck$ (which is strictly greater than $c$ also as $k\geq2$). In particular, the prime factors of $a_1$ and $a_2$ cannot contain any of the $q$'s. Hence, $e_{i1}=e_{i2}$ for all $1\leq i\leq m$ follows from unique factorization. Therefore, it must be the case that $G_1=G_2$ since $\langle g_1,\dots,g_m \rangle$ is commutative. Finally, again viewing $\langle f_1,\dots,f_n,g_1,\dots,g_m\rangle$ inside $\GL_n(\mathbb{Q})$ which is cancellative, we deduce that $F_1=F_2$ as well. Thus the map in \eqref{eq:map+decomp} is an isomorphism as claimed.        
\end{proof}
\begin{remark} Alternatively, when $(a,b,c)=(3,4,5)$ and $k\equiv 1\Mod{5}$, then one can prove the freeness of $\langle \alpha ,\beta\rangle$ by adapting the argument in the Appendix of \cite{Tao} for the matrices $A,B\in \textup{SO}(3)$, given there. Namely, the reductions in this case of the matrices $\alpha$ and $\beta$ modulo $5$ are identical to the matrices $5A$ and $5B$ in \cite{Tao}, respectively. In particular, the kernel and range calculations given in \cite{Tao} for $5A$ and $5B$ acting on $\mathbb{F}_5^3$ also hold for the matrices $\alpha$ and $\beta$. Hence, by the same logic as in \cite{Tao}, the semigroup $\langle \alpha,\beta \rangle$ is free. In fact, Tao proves (similar to \cite{swierczkowski1994class}) the stronger claim that these matrices generate a free group, not just semigroup.   
\end{remark}
\bibliographystyle{plain}
\bibliography{main}

\end{document}